\documentclass[12pt]{article}

\usepackage{float}
\restylefloat{table}

\usepackage[utf8]{inputenc}
\usepackage{placeins}
\usepackage{amsmath,amssymb,amsfonts}
\usepackage{graphicx}
\usepackage{textcomp}
\usepackage{xcolor}
\usepackage{amsthm}
\usepackage{multirow}
\usepackage{siunitx}
\usepackage{bm}
\usepackage{thmtools}
\usepackage{algorithm}
\usepackage{algpseudocode}
\usepackage{mathrsfs}
\algrenewcommand\algorithmicrequire{\textbf{Input:}}
\algrenewcommand\algorithmicensure{\textbf{Output:}}

\usepackage[english]{babel}

\usepackage[a4paper,top=2cm,bottom=2cm,left=3cm,right=3cm,marginparwidth=1.75cm]{geometry}
\usepackage{graphicx}   % a package for including graphics and images
\usepackage{tikz}   % a versatile package for doing fancy diagrams and graphs 
\usepackage{amsmath,amssymb,amsthm}  % this loads some useful maths typesetting packages
\usepackage{listings,verbatim}
\usepackage{parskip}
\usepackage{graphicx}
\usepackage{lscape}
\usepackage{tabularx}

\usepackage{bbding}
\usepackage{cases}
\usepackage{amsfonts}
\usepackage{mathrsfs}
\usepackage{enumerate}
\usepackage{datetime}
\usepackage{multirow}

\newtheorem{theorem}{Theorem}[section]   % these set up some environments for theorems, definitions, etc.

\newtheorem{lemma}[theorem]{Lemma}

\theoremstyle{definition}
\newtheorem{definition}[theorem]{Definition}
\theoremstyle{remark}

\theoremstyle{definition}

\newcommand{\rank}{\mbox{rank}}

\title{Low-Rank Reduced Biquaternion Tensor Ring Decomposition and Tensor Completion}
\bigskip

 \author{Hui Luo${}^1$,\ Xin Liu${}^2$\thanks{\em Corresponding author: xiliu@must.edu.mo (Xin Liu)}, \ Wei Liu${}^3$ and Yang Zhang${}^4$ \\[4mm]
 {\small ${}^1$ School of Computer Science and Engineering,} {\small Faculty of Innovation Engineering},\\ {\small Macau University of Science and Technology, Macau, China}\\
 {\small ${}^2$Macau Institute of Systems Engineering}, {\small Faculty of Innovation Engineering},\\ {\small Macau University of Science and Technology, Macau, China}\\ 
 {\small ${}^3$School of Artificial Intelligence at Sun Yat-sen University and the} \\ {\small   Guangdong Key Laboratory of Big Data Analysis and Processing,}  {\small Guangzhou, China}\\
 {\small ${}^4$Department of Mathematics, University of Manitoba, MB, Canada} }

\date{}

\begin{document}
\maketitle

\begin{abstract}
We define the reduced biquaternion tensor ring (RBTR) decomposition and provide a detailed exposition of the corresponding algorithm RBTR-SVD. Leveraging RBTR decomposition, we propose a novel low-rank tensor completion algorithm RBTR-TV integrating RBTR ranks with total variation (TV) regularization to optimize the process. Numerical experiments on color image and video completion tasks indicate the advantages of our method.
 \\




\end{abstract}

\section{Introduction}

In the past decades, tensor-based learning methods are extensively used in domains like computer vision, recommender systems and signal processing \cite{fanaee2016tensor,kolda2009tensor, shao2022tucker}.   Tensors' ability to seamlessly integrate multiple modes of information makes them particularly suited for the tasks like image reconstruction, image classification, object detection, signal reconstruction, and many other applications where traditional vector or matrix representations might fall short. 

During the course of collecting, storing, or transmitting data, some information may become damaged or missing.   Tensor completion arises in an effort to recover the true underlying data from these incomplete observations.  
The quaternion tensor   completion can effectively capture the relationships among color channels and has produced outstanding results in image and video completion \cite{bengua2017efficient,chen2023quaternion, he2023eigenvalues, jia2022non, pei2008eigenvalues}.  

Reduced biquaternions (RBs) $\mathbb{H}_c$ is one kind of quaternion algebras \cite{2020Least, schutte1990hypercomplex}.   RBs has some notable advantages in computational efficiency and algorithmic simplicity due to its commutativity of the multiplication.   For instance,   the  singular value decompositions (SVDs) of  RB matrices are significantly more computationally efficient than ones of Hamilton quaternion matrices, requiring only a quarter of the computational efforts \cite{pei2008eigenvalues}. In \cite{el2017color}, RBs is applied to effectively represent color images  within the domain of face recognition, significantly improving the performance over traditional PCA methods. In \cite{gai2021reduced}, a neural network based on RBs is proposed, exhibiting commendable performance in the implementation of  image denoising and image classification. 

Consequently, due to its simplicity in theoretical analyses and algorithmic implementations, alongside its extensive applications in image and video processing tasks, we choose RBs for image and video representations to solve tensor completion problems in this paper. 

One of the strategies for tensor completion is to minimize the rank of the tensor, thereby updating it to a lower-rank tensor that best represents the underlying structure \cite{huang2020robust}, which can be presented as follows: 
\begin{equation}
    \min_{\mathcal{X}}  \ \text{rank}(\mathcal{X}) \ \ 
    \text{s.t.}  \quad P_{\Omega}(\mathcal{X}) = P_{\Omega}(\mathcal{T}), 
    \label{minrank}
\end{equation}
where $\mathcal{X} \in \mathbb{H}_c^{I_1 \times I_2 \times \ldots \times I_N}$ is the recovered $N$th-order RB tensor, $\mathcal{T} \in \mathbb{H}_c^{I_1 \times I_2 \times \ldots \times I_N}$ is the observed $N$th-order RB tensor, and $P_{\Omega}(\cdot)$ represents the operator for projection on $\Omega$ which is the collection of indices of known entries. 

To tackle this issue, we first need to choose the rank of a tensor. Typically, a tensor rank can be characterized using a variety of tensor decomposition techniques, including CP decomposition, Tucker decomposition, tensor train and tensor ring decomposition \cite{zhao2016tensor}, etc. The tensor ring decomposition is a closed-loop structure decomposition, where each core tensor is connected in a ring, and its rank is defined by the dimensions of the core tensors in the ring.   Many researches have been done on tensor ring decompositions, demonstrating its superiority in certain application scenarios compared to other methods \cite{he2022hyperspectral, wu2023tensor}. 
\iffalse
{\color{blue}
While the current tensor ring decomposition methods have achieved certain effectiveness in handling incomplete and noisy data, real-number tensor ring decomposition often struggles to capture the intrinsic channel correlations in color images. This is particularly challenging when dealing with higher-order tensors, where the number of parameters and the computational complexity can significantly increase. Therefore, it needs to explore tensor ring decomposition theory based on reduced biquaternions, which  includes proving basic theorems of RBs, defining and solving RB tensor ring decomposition and completion problems, leveraging their capability in data representation to alleviate the challenges associated with large-scale data storage and processing. 
%Inspired by previous work, we will 
Our main contributions are as follows: 
\begin{enumerate}
\item
Introducing reduced biquaternion tensor ring decomposition: A novel tensor decomposition method using tensor ring structures and reduced biquaternions, lowering storage costs versus TR-SVD at comparable error levels. In image processing, it effectively preserves image quality with compressed data sizes. 
\item
Introducing reduced biquaternion tensor ring completion: A novel tensor completion method based on reduced biquaternion image representation. This approach integrates  RBTR rank with total variation regularization, effectively enhances color image completion, and offers  improved results in image processing tasks. 
\end{enumerate}
}
\fi

Current tensor ring decomposition methods based on real-number tensors may face limitations in capturing intrinsic channel correlations, especially in color images and videos. Leveraging the advantages of RBs, such as their ability to effectively represent interactions among multiple channels and reduce computational costs, reduced biquaternion-based tensor ring decomposition (RBTR) offers a promising solution to address these challenges. Therefore, it is essential to further establish the theoretical foundation of RBTR to enhance its applications in large-scale data storage and processing. Our main contributions are as follows: 
\begin{enumerate}
\item
We propose reduced biquaternion tensor ring decomposition, a novel tensor decomposition method employing tensor ring structures with reduced biquaternions. This method achieves lower storage costs compared to TR-SVD while maintaining comparable error levels. In image processing, it effectively preserves image quality with compressed data sizes. 
\item
We develop reduced biquaternion tensor ring completion, a novel tensor completion method based on reduced biquaternion image representation. By integrating RBTR rank with total variation regularization, this method achieves superior reconstruction quality in color image and video completion compared to existing methods.
\end{enumerate}

This paper is structured as follows: In Section 2, we review some definitions and properties of  reduced biquaternion matrices and tensors. In Section 3, we investigate  the reduced biquaternion tensor ring decomposition and propose an algorithm to solve this decomposition structure. Subsequently, in Section 4, we concentrate on the reduced biquaternion tensor completion problem, and  propose a novel method rooted in the RBTR decomposition. Experimental results validating the efficacy of our approach and comparing it with other existing methods are presented in Section 5. Finally, Section 6 concludes our findings and offers insights into potential future research directions.

\section{Preliminaries}

In this section, we recall some definitions and properties of reduced biquaternion matrices and tensors along with the proofs of related theorems. 
 
Reduced biquaternions (RBs) $  \mathbb{H}_c$ was introduced by Segre \cite{Segre1892}, which is an algebra over the real number field $\mathbb{R}$ with a basis $\{ 1, i, j, k \}$ and has the following form:
\begin{flalign}
 \mathbb{H}_c = \{ a + bi + cj + dk \ | \ i^2 = k^2 = -1, \ j^2 = 1, \    ij = ji = k, a, b,c, d \in \mathbb{R} \}. 
 \end{flalign}
Clearly, the imaginary units $i, j, k$ also satisfy the following  rules:
\[
  jk = kj = i, \quad ki = ik = -j. 
\]
Thus all elements in $\mathbb{H}_c$ commute, that is,  it is a commutative algebra. Comparing with Hamilton quaternions, beside the commutativity, another advantage of reduced biquaternions is having an orthogonal basis $\{e_1, e_2\}$ over the complex number field $\mathbb{C}$, which can be constructed as follows (see, e.g., \cite{pei2004commutative}): for any $ q \in \mathbb{H}_c$,
\begin{flalign}
    {q}= a+bi + cj+ dk = (a+ b i)+(c + d i)j = q_a + q_b j = q_{c1}e_1 + q_{c2}e_2, \notag
\end{flalign}
where $q_{c1} = q_a + q_b,q_{c2}=q_a - q_b$, $e_1$ and $e_2$ are defined as:
\[    
e_1= \frac{1+j}{2}, \ e_2= \frac{1-j}{2}.
\]
It is easy to see that
\[
    e_1e_2=0, \
    e_1^n = e_1^{n-1} = \dots = e_1^2=e_1, \ e_2^n = e_2^{n-1} = \dots = e_2^2=e_2. 
\]
Moreover, the conjugate  of $q$ is $\overline{{q}}=a-bi+cj-dk=\overline{{q_{c1}}}e_1+\overline{{q_{c2}}}e_2$ and its modulus is $|{{q}}|=\sqrt{a^2+b^2+c^2+d^2}$.  
 
Throughout  of this paper,  we will use a lowercase letter \(a\) for scalars, a bold lowercase \( \mathbf{a} \) for vectors, a bold uppercase \( \mathbf{A} \) for matrices, and calligraphic uppercase \( \mathcal{A} \) for tensors.  $\Re(\cdot)$ denotes the real part of a reduced biquaternion. 

For a given reduced biquaternion vector $\mathbf{q} = (q_{i}) \in \mathbb{H}_c^{n \times 1}$, the 2-norm is defined as $\|\mathbf{q}\|_2 = \sqrt{\sum_{i} |q_i|^2}$. For a given reduced biquaternion matrix $\mathbf{Q} = (q_{i,j}) \in \mathbb{H}_c^{M\times N}$, its Frobenius norm is defined as: $\lVert {\mathbf{Q}} \rVert_{F} = \sqrt{\sum_{i,j} |{{q}}_{i,j}|^2}$ and its nuclear norm is defined as $\lVert {\mathbf{Q}} \rVert_{*} = \sum_{i} |\sigma_{i}({\mathbf{Q}})|$, where $\sigma_{i}({\mathbf{Q}})$ is the $i$-th singular value of ${\mathbf{Q}}$.  The conjugate transpose $\mathbf{Q}^H$ of $\mathbf{Q}$ is defined as $\mathbf{Q}^H = (a_{i, j})^H = (\overline{a_{j, i}}) \in \mathbb{H}_c^{N\times M}$. The trace of a square RB matrix $\mathbf{A}$ is defined in a usual way,  denoted by $\mathrm{Tr}(\mathbf{A})$.

On the other hand, we can use the orthogonal basis $\{e_1, e_2\}$ to represent a reduced biquaternion matrix ${\mathbf{Q}} = \mathbf{A} + \mathbf{B} i + \mathbf{C} j + \mathbf{D} k \in \mathbb{H}_c^{M \times N}$ with  $\mathbf{A}, \mathbf{B}, \mathbf{C}, \mathbf{D} \in \mathbb{R}^{M \times N}$  as follows: 
\begin{flalign}
    {\mathbf{Q}} =(\mathbf{A} + \mathbf{B} i)+(\mathbf{C} + \mathbf{D} i)j = \mathbf{Q}_a + \mathbf{Q}_b j  
    = \mathbf{Q}_{c1} e_1 + \mathbf{Q}_{c2} e_2, 
\label{RBMR}
\end{flalign}
where $\mathbf{Q}_{c1}, \mathbf{Q}_{c2} \in \mathbb{C}^{M \times N} $ and $\mathbf{Q}_{c1} = \mathbf{Q}_a + \mathbf{Q}_b,\mathbf{Q}_{c2} = \mathbf{Q}_a - \mathbf{Q}_b$. By the definition of the conjugate transpose $\mathbf{Q}^H$, we have ${\mathbf{Q}}^H= \mathbf{Q}_{c1}^H e_1 + \mathbf{Q}_{c2}^H e_2.$
Moreover, a complex representation of  ${\mathbf{Q}}$ is given by
\begin{equation}
  \begin{bmatrix}
\mathbf{Q}_{c1} & 0 \\
0 & \mathbf{Q}_{c2}
\end{bmatrix} \in \mathbb{C}^{2M \times 2N}.
\label{equivalent complex representation}
\end{equation}
A real representation of $\mathbf{Q}=\mathbf{A}+\mathbf{B} i+\mathbf{C} j+\mathbf{D} k \in \mathbb{H}_c^{M \times N}, \ \mathbf{A}, \mathbf{B}, \mathbf{C}, \mathbf{D} \in \mathbb{R}^{M \times N}$ is given by 
\[
\mathbf{Q}^{R}=\left[\begin{array}{rrrr}
\mathbf{A} & -\mathbf{B} & \mathbf{C} & -\mathbf{D} \\ 
\mathbf{B} & \mathbf{A} & \mathbf{D} & \mathbf{C} 
\\ \mathbf{C} & -\mathbf{D} & \mathbf{A} & -\mathbf{B} 
\\ \mathbf{D} & \mathbf{C} & \mathbf{B} & \mathbf{A}
\end{array}\right].
\]
%(See, for example, \cite{pei2004commutative}).
\begin{lemma}
(Reduced biquaternion singular value decomposition RBSVD)\cite{pei2008eigenvalues} \\
With the notations in the formula \eqref{RBMR}, if the SVDs of  $\mathbf{Q}_{c1}$ and $\mathbf{Q}_{c2}$ are in the following forms: 
\[ 
\mathbf{Q}_{c1} = \mathbf{U}_1 \boldsymbol{\Sigma}_1 \mathbf{V}_1^\mathrm{H}, \quad 
 \mathbf{Q}_{c2} = \mathbf{U}_2 \boldsymbol{\Sigma}_2 \mathbf{V}_2^\mathrm{H}, 
\]
then the SVD of ${\mathbf{Q}} $ is
\begin{flalign}
	{\mathbf{Q}} = {{\mathbf{U}}} \boldsymbol{\Sigma} {{\mathbf{V}}}^\mathrm{H},
\end{flalign}
  where ${\mathbf{U}} =  \mathbf{U}_1 e_1 +  \mathbf{U}_2 e_2 $, $\boldsymbol{\Sigma} =  \boldsymbol{\Sigma}_1 e_1 +  \boldsymbol{\Sigma}_2 e_2 $, ${\mathbf{V}} =  \mathbf{V}_1 e_1 +  \mathbf{V}_2 e_2$. 
\end{lemma}

There are several different definitions of the ranks of matrices over commutative rings. For our purpose, using above RBSVD, we  define   {\it the rank of a RB matrix} $\mathbf{A}$ to be the number of its non-zero singular values, that is, $\rank (\mathbf{A})= \rank(\mathbf{\Sigma}).$   Next, we will derive some useful results about the rank of RB matrix $\mathbf{A}$. 

\begin{theorem}\label{4rank}
 Assume that $q=a+bi+cj+dk\neq 0 \in \mathbb{H}_c$. Then $\rank (q^R)=4.$  
\end{theorem}
\begin{proof}
For $q\in \mathbb{H}_c$, its real representation is given by 
\[
q^{R}=\left[\begin{array}{rrrr}
a & -b & c & -d \\ 
b & a & d & c 
\\ c& -d & a & -b 
\\ d & c & b & a
\end{array}\right].
\]
Now we show that each two columns of $q^R$ are  linearly independently. Support that there exists a nonzero real number $k$ such that 
\[
\left[\begin{array}{r}
a  \\ 
b\\
c  \\
d
\end{array}\right]=k\left[\begin{array}{r}
-b  \\ 
a\\
-d  \\
c
\end{array}\right],
\]
that is, 
\begin{equation}\label{a1234}
a=-kb,\ b=ka,\ c=-kd,\ d=kc.
\end{equation} 
Thus we have 
$(1+k^2)a=0,\ (1+k^2)c=0$, which imply
 $a=0,\ c=0.$
According to the relations in \eqref{a1234},  we obtain $b=0, d=0,$ and then $q=0$, a contradiction to the assumption $q=a+bi+cj+dk\neq 0$. Hence, we proved the first and second column of $q^R$ are linearly independently. Similarly, we can show that any two columns of $q^R$ are also linearly independently. Therefore, if $q=a+bi+cj+dk\neq 0$,  then $\rank(q^R)=4$. 
 \end{proof}

Some of the following  properties of the real representations are well-known (see, e.g., \cite{li2021}). 
\begin{lemma}
\label{pro22}
Let $\mathbf{A}, \mathbf{B} \in \mathbb{H}_c^{M \times N}, \ \mathbf{C} \in \mathbb{H}_c^{N \times S}, \ a \in \mathbb{R}$. Then:
\begin{itemize}
\item[(a)] $(\mathbf{A}+\mathbf{B})^{R}=\mathbf{A}^{R}+\mathbf{B}^{R},\ (a \mathbf{A})^{R}=a \mathbf{A}^{R};$
\item[(b)] $(\mathbf{A} \mathbf{C})^{R}=\mathbf{A}^{R} \mathbf{C}^{R}$;
\item[(c)] $(\mathbf{A}^H )^{R}={(\mathbf{A}^{R})}^{T}$;
\item[(d)]$\mathbf{U}$ is unitary if and only if $\mathbf{U}^R$ is orthogonal;
\item[(e)] $\rank(\mathbf{A}^R)=4\cdot \rank(\mathbf{A})$;
\item[(f)]$\rank(\mathbf{AB}) \leq min\{\rank(\mathbf{A}), \rank(\mathbf{B})\}.$
\end{itemize}
\end{lemma}

\begin{proof}
(a)-(c) are known in \cite{li2021}. We only need to prove (d)-(f).
To show (d), applying the real representation operator on the $\mathbf{U}\mathbf{U}^H=\mathbf{I}_n$ gives $\mathbf{U}^R(\mathbf{U}^H)^R={\mathbf{I}_n}^R$. Using (c), we get  $\mathbf{U}^R(\mathbf{U}^R)^T=\mathbf{I}_{4n}.$  

To prove (e),  let the SVD of $\mathbf{A}=\mathbf{U\Sigma}\mathbf{ V}^H$. By (b), (c) and (d), we have 
\[
\mathbf{A}^R=\mathbf{U}^R\mathbf{\Sigma}^R (\mathbf{V}^H)^R=\mathbf{U}^R\mathbf{\Sigma}^R (\mathbf{V}^R)^T.
\]
Again, by (d), $\mathbf{U}^R$ and $ (\mathbf{V}^R)^T$ are orthogonal matrices, we obtain
\[
\rank(\mathbf{A}^R) = \rank(\mathbf{\Sigma}^R).
\]
Next, we prove $\rank(\mathbf{\Sigma}^R)=4\cdot \rank(\mathbf{\Sigma})$.
Without loss of generalization and for simplicity, we assume that $\mathbf{\Sigma}$ is a $2 \times 2$ matrix. According to \cite{pei2008eigenvalues}, $\mathbf{\Sigma}$ should be in the form of 
\[
\mathbf{\Sigma}=\left[\begin{array}{rr}
\sigma_1 & 0 \\
0 & \sigma_2
\end{array}\right]=\left[\begin{array}{rr}
a_{1} & 0 \\
0 & a_{2}
\end{array}\right]+\left[\begin{array}{rr}
c_{1} & 0 \\
0 & c_{2}
\end{array}\right]j, \ a_{1}, a_{2}, c_{1}, c_{2} \in \mathbb{R}. 
\] 
Then 
\[
\mathbf{\Sigma}^R=\left[\begin{array}{rrrrrrrr}
a_{1} & 0 & 0 & 0& c_{1} & 0 & 0 & 0\\ 
0 & a_{2} & 0 & 0& 0 & c_{2} & 0 & 0\\ 
0 & 0 & a_{1} & 0& 0 & 0 & c_{1} & 0\\
0 & 0 & 0 & a_{2}& 0 & 0 & 0& c_{2}\\ 
c_{1} & 0 &0 & 0 & a_{1} & 0 & 0 & 0\\ 
0 & c_{2} & 0 & 0& 0 & a_{2} & 0 & 0\\
0 & 0 &c_{1}& 0&  0&0 & a_{1} & 0\\ 
0 & 0 & 0 & c_{2}& 0 & 0 & 0 & a_{2}\\ 
\end{array}\right].
\]
Upon some arrangements of rows and columns
\[
\mathbf{\Sigma}^R  \rightarrow\left[\begin{array}{rrrrrrrr}
a_{1}  & 0 &c_{1}  & 0&0&0&0&0 \\ 
0  & a_{1} & 0  & c_{1}&0&0&0&0 \\
c_{1}  &0  & a_{1} & 0&0&0&0&0 \\ 
0 & c_{1}& 0 & a_{1}&0&0&0&0\\ 
0&0&0&0&a_{2}  & 0 & c_{2}  & 0\\ 
0&0&0&0& 0  & a_{2} & 0 & c_{2}\\
0&0&0&0&c_{2} &0 & a_{2} & 0 \\ 
0&0&0&0&0  &c_{2}&  0 & a_{2}\\ 
\end{array}\right]=\left[\begin{array}{rr}
\mathbf{\sigma}_1^R & 0 \\
0 & \mathbf{\sigma}_2^R
\end{array}\right],
\]
where $\mathbf{\sigma}_1=a_{1}+c_{1}j, \mathbf{\sigma}_2=a_{2}+c_{2}j.$
Obviously, if $\mathbf{\sigma}_1\neq 0, \mathbf{\sigma}_2\neq 0,$ then by Theorem \ref{4rank}, we get $\rank(\mathbf{\Sigma}^R) = 8,$ i.e., $\rank(\mathbf{\Sigma}^R)= 4 \cdot \rank(\mathbf{\Sigma}).$ Therefore, by the definition of rank of $\mathbf{A}$,   
\[
\rank(\mathbf{A}) = \rank(\mathbf{\Sigma}) =\frac{1}{4}\cdot 
 \rank(\mathbf{\Sigma}^R) = \frac{1}{4}\cdot \rank(\mathbf{A}^R).\]

Basing on (e), we can show (f) as follows. 

By (a) and (e), 
\[
\rank(\mathbf{AB})=\frac{1}{4}\cdot \rank((\mathbf{AB})^R) = \frac{1}{4} \cdot \rank(\mathbf{A}^R\mathbf{B}^R).
\]
Moreover, for the real matrices $\mathbf{A}^R, \ \mathbf{B}^R$,
\[
\rank(\mathbf{A}^R\mathbf{B}^R) \leq \min \{\rank(\mathbf{A}^R), \ \rank(\mathbf{B}^R)\}= 4 \cdot \min \{\rank(\mathbf{A}), \ \rank(\mathbf{B})\}.\]
Therefore
\[
\rank(\mathbf{AB}) \leq \min\{\rank(\mathbf{A}), \ \rank(\mathbf{B})\}.\]
\end{proof}

Next, we verify that the norm for RB vectors has the similar properties as the norm for real number vectors. We will use these properties in next sections.
\begin{theorem}
\label{lemma2.4}
Let $\mathbf{x}, \mathbf{y} \in \mathbb{H}_c^{n\times 1}$, where $ \mathbf{x} = \mathbf{x}_0 + \mathbf{x}_1i + \mathbf{x}_2j + \mathbf{x}_3k $ and $ \mathbf{y} = \mathbf{y}_0 + \mathbf{y}_1i + \mathbf{y}_2j + \mathbf{y}_3k $ with $ \mathbf{x}_l, \ \mathbf{y}_l\in \mathbb{R}^{n\times 1} (l=1,2,3,4)$. Then
\begin{itemize}
\item[(a)] $\|\mathbf{x} - \mathbf{y}\|_2^2 = \|\mathbf{x}\|_2^2 + \|\mathbf{y}\|_2^2 - 2 \Re(\mathbf{x}^H\mathbf{y})$;
\item[(b)] $\Re(\mathbf{x}^H\mathbf{y}) \leq \|\mathbf{x}\|_2 \|\mathbf{y}\|_2$.
\end{itemize}
\end{theorem}

\begin{proof}
For part (a), we have
\begin{align*}
    \|\mathbf{x} - \mathbf{y}\|_2^2 &= \|\mathbf{x}_0 - \mathbf{y}_0\|_2^2 + \|\mathbf{x}_1 - \mathbf{y}_1\|_2^2 + \|\mathbf{x}_2 - \mathbf{y}_2\|_2^2 + \|\mathbf{x}_3 - \mathbf{y}_3\|_2^2 \\
    &= \|\mathbf{x}_0\|_2^2 + \|\mathbf{y}_0\|_2^2 - 2\mathbf{x}_0^T\mathbf{y}_0 + \|\mathbf{x}_1\|_2^2 + \|\mathbf{y}_1\|_2^2 - 2\mathbf{x}_1^T\mathbf{y}_1 \\
    &\quad + \|\mathbf{x}_2\|_2^2 + \|\mathbf{y}_2\|_2^2 - 2\mathbf{x}_2^T\mathbf{y}_2 + \|\mathbf{x}_3\|_2^2 + \|\mathbf{y}_3\|_2^2 - 2\mathbf{x}_3^T\mathbf{y}_3 \\
    &= \|\mathbf{x}\|_2^2 + \|\mathbf{y}\|_2^2 - 2(\mathbf{x}_0^T\mathbf{y}_0 + \mathbf{x}_1^T\mathbf{y}_1 + \mathbf{x}_2^T\mathbf{y}_2 + \mathbf{x}_3^T\mathbf{y}_3), 
\end{align*}
and
\begin{align*}
    \Re(\mathbf{x}^H \mathbf{y}) &= \Re((\mathbf{x}_0 + \mathbf{x}_1i + \mathbf{x}_2j + \mathbf{x}_3k)^H(\mathbf{y}_0 + \mathbf{y}_1i + \mathbf{y}_2j + \mathbf{y}_3k)) \\
    &= \Re((\mathbf{x}_0 - \mathbf{x}_1i + \mathbf{x}_2j - \mathbf{x}_3k)^T (\mathbf{y}_0 + \mathbf{y}_1i + \mathbf{y}_2j + \mathbf{y}_3k)) \\
    &= \mathbf{x}_0^T\mathbf{y}_0 + \mathbf{x}_1^T\mathbf{y}_1 + \mathbf{x}_2^T\mathbf{y}_2 + \mathbf{x}_3^T\mathbf{y}_3.
\end{align*}
Thus
\[
    \|\mathbf{x} - \mathbf{y}\|_2^2 = \|\mathbf{x}\|_2^2 + \|\mathbf{y}\|_2^2 - 2\Re(\mathbf{x}^H \mathbf{y}),
\]
which completes the proof of part (a).

For part (b),  let $\mathbf{M} = \begin{bmatrix} \mathbf{x}_0 \\ \mathbf{x}_1 \\ \mathbf{x}_2 \\ \mathbf{x}_3 \end{bmatrix}, \ \mathbf{N} = \begin{bmatrix} \mathbf{y}_0 \\ \mathbf{y}_1 \\ \mathbf{y}_2 \\ \mathbf{y}_3 \end{bmatrix} \in \mathbb{R}^{4n \times 1}.$ Then, applying the Cauchy-Schwarz inequality
\begin{align*}
|\mathbf{M}^T \mathbf{N}| &\leq \|\mathbf{M}\|_2 \|\mathbf{N}\|_2, \\
\end{align*}
we obtain
\begin{flalign*}
|\Re(\mathbf{x}^H\mathbf{y})| = |\mathbf{M}^T \mathbf{N}| \leq \|\mathbf{x}\|_2 \|\mathbf{y}\|_2. 
\end{flalign*}
\end{proof}

\begin{theorem}
\label{min2norm}
Given the parameters $\beta>0, \lambda>0$, and let $\mathbf{x}, \mathbf{y} \in 
\mathbb{H}_c^{n\times 1}$.  The closed-form solution to the minimization problem 
\begin{flalign*}
\min_{\mathbf{x} \in \mathbb{H}_c^{n \times 1}} f(\mathbf{x}) = \frac{\beta}{2} \|\mathbf{x} - \mathbf{y}\|_2^2 + \lambda \|\mathbf{x}\|_2
\end{flalign*}
is given by
\begin{flalign*}
\hat{\mathbf{x}} = \max \left\{\|\mathbf{y}\|_2 - \frac{\lambda}{\beta}, \ 0 \right\} \frac{\mathbf{y}}{{\|\mathbf{y}\|_2}}.
\end{flalign*}
\end{theorem}

\begin{proof}
We consider two cases based on the value of $\|\mathbf{y}\|_2$ relative to $\frac{\lambda}{\beta}$:

\textbf{Case 1:}  $\|\mathbf{y}\|_2 \leq \frac{\lambda}{\beta}$. We will prove that $\hat{\mathbf{x}} = \mathbf{0}$. 

For all $\mathbf{x} \neq \mathbf{0}$, it follows from  Theorem \ref{lemma2.4} that
\begin{align*}
f(\mathbf{x}) - f(\mathbf{0}) &=\frac{\beta}{2} \|\mathbb{\mathbf{x}} - \mathbf{y}\|_2^2 + \lambda \|\mathbf{x}\|_2 - \frac{\beta}{2} \|\mathbf{y}\|_2^2 \\
&= \frac{\beta}{2} (\|\mathbf{x}\|_2^2 + \|\mathbf{y}\|_2^2 - 2\Re(\mathbf{x}^H\mathbf{y})) + \lambda \|\mathbf{x}\|_2 - \frac{\beta}{2} \|\mathbf{y}\|_2^2 \\
&= \frac{\beta}{2} \|\mathbf{x}\|_2^2 - \beta\Re(\mathbf{x}^H\mathbf{y}) + \lambda \|\mathbf{x}\|_2.
\end{align*}
As proven in Theorem \ref{lemma2.4} part (b), we have  $|\Re(\mathbf{x}^H\mathbf{y})| \leq \|\mathbf{x}\|_2 \|\mathbf{y}\|_2$, and thus
\begin{align*}
f(\mathbf{x}) - f(\mathbf{0}) &\geq  \frac{\beta}{2} \|\mathbf{x}\|_2^2 - \beta\|\mathbf{x}\|_2 \|\mathbf{y}\|_2 + \lambda \|\mathbf{x}\|_2 \\
&\geq \frac{\beta}{2} \|\mathbf{x}\|_2^2 + \|\mathbf{x}\|_2 (\lambda - \beta\|\mathbf{y}\|_2).
\end{align*}
Note that $\lambda - \beta\|\mathbf{y}\|_2 \geq 0$, and thus $f(\mathbf{x}) - f(\mathbf{0}) > 0$, which implies that $\hat{\mathbf{x}} = \mathbf{0}$.

\textbf{Case 2:}  $\|\mathbf{y}\|_2 > \frac{\lambda}{\beta}$. We will  find the  $\hat{\mathbf{x}}$ which minimizes $f(\mathbf{x})$. 

By setting the gradient of $f(\mathbf{x})$ to be zero, we have
\begin{align*}
\beta(\mathbf{x} - \mathbf{y}) + \lambda \frac{\mathbf{x}}{\|\mathbf{x}\|_2} = 0. 
\end{align*}
Solving for $\mathbf{x}$ yields
\begin{align}
\mathbf{x}\left(\beta + \frac{\lambda}{\|\mathbf{x}\|_2}\right) - \beta\mathbf{y} = 0.
\label{tag1}
\end{align}
Now, taking the 2-norm of both sides gives
\begin{align*}
\|\mathbf{x}\|_2 \left(\beta + \frac{\lambda}{\|\mathbf{x}\|_2}\right) - \beta\|\mathbf{y}\|_2 = 0, 
\end{align*}
and thus
\begin{align}
\|\mathbf{x}\|_2 = \|\mathbf{y}\|_2 - \frac{\lambda}{\beta}.
\label{tag2}
\end{align}
Substituting \eqref{tag2} into \eqref{tag1}, we obtain
\begin{align*}
\mathbf{x}\left(\beta + \frac{\lambda}{\|\mathbf{y}\|_2 - \frac{\lambda}{\beta}}\right) - \beta\mathbf{y} = 0  , 
\end{align*}
from which we have
\begin{align*}
\hat{\mathbf{x}} = \frac{\|\mathbf{y}\|_2 - \frac{\lambda}{\beta}}{\|\mathbf{y}\|_2}\mathbf{y} . 
\end{align*}
\end{proof}

\section{Reduced biquaternion tensor decomposition}
First, we define the reduced biquaternion tensor ring decomposition, extending from the real-valued framework \cite{zhao2016tensor}. 

\begin{definition}
 (Reduced Biquaternion Tensor Ring (RBTR) Decomposition)\\
 Let $\mathcal{T} \in \mathbb{H}_c^{I_1 \times I_2 \times \ldots \times I_N}$. The reduced biquaternion tensor ring (RBTR) decomposition is to break down $\mathcal{T}$  into a series of latent tensors $\mathcal{Z}_k \in \mathbb{H}_c^{r_k \times I_k \times r_{k+1}}$ for $k=1,2,\ldots,N$. Each entry of $\mathcal{T}$ is defined through:
\begin{flalign}
    \mathcal{T}(i_1,i_2,\dots,i_N) = \mathrm{Tr}\{\mathbf{Z}_1(i_1)\mathbf{Z}_2(i_2) \cdots \mathbf{Z}_N(i_N)\}, \notag
\end{flalign}
where $\mathbf{Z}_k(i_k) \in \mathbb{H}_c^{r_k \times r_{k+1}}$ indicates the $i_{k}$-th lateral slice of the latent tensor $\mathcal{Z}_k$, and $r_{N+1}=r_{1}$. $\mathbf{r}=[r_1,r_2,\dots,r_N] $ is a vector called {\it RBTR rank} of $\mathcal{T}$. 

For simplicity's sake,  we refer to the RBTR decomposition as $\mathcal{T} = \mathfrak{TR}(\mathcal{Z}_1,\mathcal{Z}_2,\cdots, \mathcal{Z}_N)$. 
\end{definition}

Because the reduced biquaternions  satisfy multiplicative commutativity, the definition and related proofs are simpler than Hamilton quaternions, like 
 the property of Invariance of Circular Dimensional Permutation in Theorem \ref{Th1}. 
 
\begin{theorem}
\label{Th1}
(Invariance of Circular Dimensional Permutation) \\
 Let $\mathcal{T}\in\mathbb{H}_c^{I_1 \times I_2 \times \ldots \times I_N}$ be a RB tensor with a RBTR format $\mathfrak{TR}(\mathcal{Z}_1,\mathcal{Z}_2,\cdots, \mathcal{Z}_N)$. If we define $ \overleftarrow{\mathcal{T}}^k \in \mathbb{H}_c^{I_{k+1} \times I_{k+2} \times \cdots \times I_N \times I_1 \times I_2 \times \cdots \times I_k} $
 as moving the dimension of $\mathcal{T}$ circularly by $k$ steps, then we have $ \overleftarrow{\mathcal{T}}^k = \mathfrak{TR}(\mathcal{Z}_{k+1}, \cdots, \mathcal{Z}_N, \mathcal{Z}_1, \cdots, \mathcal{Z}_k).$

\end{theorem}
\begin{proof}

Note that $\mathbb{H}_c$ is a commutative ring.  For any  RB matrices $\mathbf{P} = (p_{ij}) \in \mathbb{H}_c^{M \times N}, \ \mathbf{Q} = (q_{ij}) \in \mathbb{H}_c^{N\times M}$, we have the following same trace property as real matrices.
\[
\mathrm{Tr}(\mathbf{P}\mathbf{Q}) = \sum_{i=1}^M \sum_{j=1}^N p_{ij} q_{ji}  = \sum_{j=1}^N \sum_{i=1}^M q_{ji}p_{ij} = \mathrm{Tr}(\mathbf{Q}\mathbf{P}).
\]
Clearly, this property is also true for the product of multiple RB matrices. Therefore, we can easily prove the following process
\begin{flalign*}
    &\overleftarrow{\mathcal{T}}^k (i_{k+1}, \ldots, i_{N}, i_{1}, \ldots, i_{k}) \\
     &=\mathcal{T} (i_{1}, i_{2}, \ldots, i_{N})   \\
     &= \mathrm{Tr}\{\mathbf{Z}_{1}(i_{1}) \cdots \mathbf{Z}_{k}(i_{k}) \mathbf{Z}_{k+1}(i_{k+1}) \cdots \mathbf{Z}_{N}(i_{N})  \} \\
     &= \mathrm{Tr}\{\mathbf{Z}_{k+1}(i_{k+1}) \cdots \mathbf{Z}_{N}(i_{N}) \mathbf{Z}_{1}(i_{1}) \cdots \mathbf{Z}_{k}(i_{k}) \} .  
\end{flalign*}
The above proof demonstrates that moving the dimension of \(\mathcal{T}\) circularly by \(k\) steps retains the RBTR format and 
$ \overleftarrow{\mathcal{T}}^k = \mathfrak{TR}(\mathcal{Z}_{k+1}, \cdots, \mathcal{Z}_N, \mathcal{Z}_1, \cdots, \mathcal{Z}_k).$
\end{proof}

Next, we merge the cores of $\mathcal{T} = \mathfrak{TR}(\mathcal{Z}_1,\mathcal{Z}_2,\cdots, \mathcal{Z}_N)\in \mathbb{H}_c^{I_1 \times I_2 \times \ldots \times I_N}$  in the following two ways. The new cores are called the {\it RB subchain tensors}.

The first way is to merge the first $k$ cores $\mathcal{Z}_1 \ldots \mathcal{Z}_{k}$ into a new core $\mathcal{Z}^{\leq k} \in \mathbb{H}_c^{r_1 \times \prod_{j=1}^{k} I_j \times r_{k+1}}$ whose lateral slice matrices are described as:
\begin{equation}
\begin{aligned}
    \mathbf{Z}^{\leq k}(:,\overline{i_1 \ldots i_{k}},:) & = \prod_{j=1}^{k} \mathbf{Z}_j(i_j), \notag 
\end{aligned}
\end{equation}
where $\overline{i_1 \ldots i_{k}} = i_1 + (i_2 - 1)I_1 + (i_3 - 1)I_1I_2 + \ldots + (i_k - 1) \prod_{j=1}^{k-1}I_j $. 

The second way is, in a similar way, to merge the last $N-k$ cores $\mathcal{Z}_{k+1} \ldots \mathcal{Z}_{N}$ into a new core $\mathcal{Z}^{> k}\in \mathbb{H}_c^{r_{k+1} \times \prod_{j=k+1}^{N} I_j \times r_{1}}$, whose lateral slice matrices are described as
\begin{equation}
\begin{aligned}
    \mathbf{Z}^{> k}(:,\overline{i_{k+1} \ldots i_{N}},:) & = \prod_{j={k+1}}^N \mathbf{Z}_j(i_j). \notag 
\end{aligned}
\end{equation}
where $\overline{i_{k+1} \ldots i_{N}} = i_{k+1} + (i_{k+2} - 1)I_{k+1} + (i_{k+3} - 1)I_{k+1}I_{k+2} + \ldots + (i_{N} - 1) \prod_{j=1}^{N-1}I_j $.

In the rest of this paper, we will use the following three common unfolding methods for reduced biquaternion tensors, which are extended from those for real-valued tensors \cite{kolda2009tensor}. As usual, we will use three kinds of brackets for these three unfolding methods like $\mathbf{T}_{(k)}, \mathbf{T}_{[k]}$ and $\mathbf{T}_{\langle k \rangle}$.

{\bf Classical Mode-$k$ Unfolding:}
For ${\mathcal{T}} \in \mathbb{H}_c^{I_1 \times I_2 \times \ldots \times I_N }$, let  ${\mathbf{T}}_{(k)} \in \mathbb{H}_c^{I_k \times \prod_{l \neq k} I_{l}}$ be the classical mode-k unfolding of  ${\mathcal{T}}$. The tensor element indexed by $(i_{1},i_{2},\dots ,i_{N})$ of $\mathcal{T}$ maps to the matrix element at position $(i_k,j)$-th of ${\mathbf{T}}_{(k)}$, i.e., 
\begin{flalign}
    {\mathcal{T}}(i_1,i_2,\dots,i_N) = {\mathbf{T}}_{(k)}(i_k,j), \notag
\end{flalign}
where $j = i_1 + (i_2-1)I_1 + \dots + (i_{k-1}-1) \prod_{l=1}^{k-2} I_l + (i_{k+2}-1)\prod_{l=1,l \neq k}^{k+1} I_l + \dots + (i_{N}-1) \prod_{l=1,l \neq k}^{N-1}I_l. $

{\bf Mode-$k$ Unfolding:} 
For ${\mathcal{T}} \in \mathbb{H}_c^{I_1 \times I_2 \times \ldots \times I_N }$, let  ${\mathbf{T}}_{[k]} \in \mathbb{H}_c^{I_k \times \prod_{l \neq k} I_{l}}$ be  the mode-k unfolding of ${\mathcal{T}}$. The tensor element indexed by $(i_{1},i_{2},\dots,i_{N})$ of $\mathcal{T}$ maps to the matrix element at position $(i_k,j)$-th of ${\mathbf{T}}_{[k]}$, i.e., 
\begin{flalign}
    {\mathcal{T}}(i_1,i_2,\dots,i_N) = {\mathbf{T}}_{[k]}(i_k,j), \notag
\end{flalign}
where $j=i_{k+1}+(i_{k+2}-1)I_{k+1}+\dots+ (i_{N}-1)\prod_{l=k+1}^{N-1} I_l + \dots+(i_{k-1}-1)\prod_{l=k+1}^{k-2} I_l. $

{\bf $k$-mode  Unfolding:} 
For  ${\mathcal{T}} \in \mathbb{H}_c^{I_1 \times I_2 \times \ldots \times I_N }$, let ${\mathbf{T}}_{\langle k \rangle} \in \mathbb{H}_c^{\prod_{l=1}^{k}{I_{l} \times \prod_{l=k+1}^{N} {I_{l}}} }$ be the k-mode  unfolding of  ${\mathcal{T}}$. The tensor element indexed by $(i_{1},i_{2},\dots ,i_{N})$ of $\mathcal{T}$ maps to the matrix element at position $(i,j)$-th of ${\mathbf{T}}_{\langle k \rangle}$, i.e., 
\begin{flalign}
    {\mathcal{T}}(i_1,i_2,\dots,i_N) = {\mathbf{T}}_{\langle k \rangle}(i,j), \notag
\end{flalign}
where $i = i_1 + (i_2 - 1)I_1 + \dots + (i_{k} - 1) \prod_{l=1}^{k-1} I_l,\ j=i_{k+1}+(i_{k+2}-1)I_{k+1} + \dots + (i_{N} - 1) \prod_{l=k+1}^{N-1} I_l. $
 
The three unfolding methods have the following relation in term of subchain tensors.

\begin{theorem}\label{tk}
Let $\mathcal{T} = \mathfrak{TR}(\mathcal{Z}_1,\mathcal{Z}_2,\cdots, \mathcal{Z}_N) \in \mathbb{H}_c^{I_1 \times I_2 \times \ldots \times I_N}$ represent a RB tensor  structured in RBTR format. Then its $k$-mode unfolding  $\mathbf{T}_{\langle k \rangle}$ can be decomposed by using (classical) mode-k unfoldings of RB subchain tensors, that is,   
\begin{equation}
    \mathbf{T}_{\langle k \rangle} = \mathbf{Z}^{\leq k}_{(2)} \ (\mathbf{Z}^{>k}_{[2]})^T. 
\label{Tk}
\end{equation}
\end{theorem}
\begin{proof}
By the definition of $\mathbf{T}_{\langle k \rangle}$, 
\begin{flalign}
\label{v}
    \mathbf{T}_{\langle k \rangle}(i,j) &=\mathcal{T}(i_1,i_2,\dots,i_N) \notag \\[3mm] 
    &=\mathrm{Tr} \left\{ \mathbf{Z}_1(i_1) \mathbf{Z}_2(i_2) \cdot\cdot\cdot \mathbf{Z}_N(i_N) \right\}  \notag \\ 
    &= \mathrm{Tr} \left\{ \prod_{j=1}^{k} \mathbf{Z}_j(i_j) \prod_{j=k+1}^{N} \mathbf{Z}_j(i_j) \right\}.  \notag \\ 
\end{flalign}
Recall that  the vector operator $\mathrm{Vec}( \cdot )$ is the column vector 
obtained by stacking the columns of the matrix in order. By $\text{Tr}(\mathbf{A}\mathbf{B}) = (\mathrm{Vec}(\mathbf{A}))^\top \mathrm{Vec}(\mathbf{B}^\top)$, (\ref{v}) becomes
\begin{flalign*}
    \mathbf{T}_{\langle k \rangle}(i,j) 
    & = \left( \mathrm{Vec} \left( \prod_{j=1}^{k} \mathbf{Z}_j(i_j) \right)\right)^\top  \mathrm{Vec} \left( \prod_{j=N}^{k+1} \mathbf{Z}^T_j(i_j) \right) \notag \\
    &= \sum_{p=1}^{r_1 r_{k+1}} \mathbf{Z}_{(2)}^{\leq k}(i,p) \left( \mathbf{Z}_{[2]}^{> k} \right)^\top (p,j),
\end{flalign*}

where $i = \overline{i_1 i_2 \ldots i_k}$ and $j = \overline{i_{k+1} i_{k+2} \ldots i_N}$.  From the formula of $\mathbf{T}_{\langle k \rangle}(i,j)$, we have
\begin{equation}
    \mathbf{T}_{\langle k \rangle} = \mathbf{Z}^{\leq k}_{(2)} \ (\mathbf{Z}^{>k}_{[2]})^\top. \notag
\end{equation}  
\end{proof}
Initially, we utilize the truncated SVD technique to extract the principal characteristics  of the tensor's 1-mode unfolding matrix $\mathbf{T}_{\langle 1 \rangle} \in \mathbb{H}_c^{I_1 \times \prod_{l=2}^{N}{I_{l} }}$. First, a threshold $\delta_1 = \sqrt{2} \epsilon \frac{\| \mathcal{T} \|_F}{\sqrt{N}}
$ is set to determine the rank of the approximation, as described in \cite{zhao2016tensor}. In this context, $\epsilon$ is the predefined tolerance level. Consequently, \( \mathbf{T}_{\langle 1 \rangle} \) can be approximated as a low-rank matrix by retaining only the singular values that exceed the threshold \( \delta_1 \): 
\begin{equation}
\mathbf{T}_{\langle 1 \rangle} = \mathbf{U}_{1} \mathbf{\Sigma}_{1} \mathbf{V}_{1}^H + \mathbf{E}_{1}, 
\label{T1}
\end{equation}
where $\mathbf{U}_{1} \in \mathbb{H}_c^{I_1 \times r_1r_2} $ and $\mathbf{\Sigma}_{1} \in \mathbb{H}_c^{r_1r_2 \times r_1r_2}$ and  $\mathbf{V}_{1}^H \in \mathbb{H}_c^{r_1r_2 \times \prod_{l=2}^{N}{I_{l} }}.$

By Theorem \ref{tk}, $\mathbf{T}_{\langle 1 \rangle}$ can be rewritten as 

\[
\mathbf{T}_{\langle 1 \rangle} = \mathbf{Z}^{\leq 1}_{(2)} (\mathbf{Z}^{>1}_{[2]})^\top. 
\]
Consider that $\mathbf{Z}^{\leq 1}_{(2)}$ is equivalent to $\mathbf{U}_1$ and $ (\mathbf{Z}^{>1}_{[2]})^T$  corresponds to \( \mathbf{\Sigma}_1 \mathbf{V}_1^H \). Using these, the primary core tensor \( \mathcal{Z}_1 \) of size $r_1 \times I_1 \times r_2$ is derived by appropriately reorganizing and switching the elements of \( \mathbf{U}_1 \). Meanwhile, $ \mathcal{Z}^{>1} $ of size $r_2 \times \prod_{j=2}^{N} I_j \times r_1 $ is derived by appropriately reorganizing and switching the elements of  \( \mathbf{\Sigma}_1 \mathbf{V}_1^H \). We can reshape  $ \mathcal{Z}^{>1} $ into matrix  $ \mathbf{Z}^{>1} $  of size $r_2I_2 \times \prod_{j=3}^{N} I_j r_1 $, and continue to apply truncated SVD by using  $\delta_2= \epsilon \frac{\| \mathcal{T} \|_F}{\sqrt{N}}$ on it:
\begin{equation}
\mathbf{Z}^{>1} = \mathbf{U}_{2} \mathbf{\Sigma}_{2} \mathbf{V}_{2}^H + \mathbf{E}_{2}, \notag
\end{equation}
where $\mathbf{U}_{2} \in \mathbb{H}_c^{r_2I_2 \times r_3}$, $\mathbf{\Sigma}_{2} \in \mathbb{H}_c^{r_3 \times r_3}$ and  $\mathbf{V}_{2}^H \in \mathbb{H}_c^{r_3 \times \prod_{j=3}^{N} I_j r_1} $. Similarly, the second core tensor \( \mathcal{Z}_2 \) of size $r_2 \times I_2 \times r_3$ is derived by appropriately reorganizing and switching the elements of \( \mathbf{U}_2 \). Following the same method, we can get the remaining cores  \( \mathcal{Z}_k \) by setting  $\delta_k= \epsilon \frac{\| \mathcal{T} \|_F}{\sqrt{N}}$, for $k>1$.

In the following, $\textbf{Algorithm 1}$ specifies the RBTR-SVD algorithm's detailed process.
\begin{figure}[!ht]
      \centering  % Centers the algorithm in the figure
    \begin{algorithm}[H]
        \caption{RBTR-SVD}
        \begin{algorithmic}[1]
            \Require 
                A RB tensor $\mathcal{T} \in {\mathbb{H}_c}^{I_{1}\times I_{2}\times \dots \times I_{N}}$ and the tolerance level $\epsilon$. 
            \Ensure Cores ${\mathcal{Z}_k}  \ (k = 1, \ldots, N)$ and the RBTR ranks\\
            Set the truncation threshold: for $k = 1$, $\delta_1 = \sqrt{2} \epsilon \frac{\| \mathcal{T} \|_F}{\sqrt{N}}$, and for $k > 1$, $\delta_k = \epsilon \frac{\| \mathcal{T} \|_F}{\sqrt{N}}$. \\

            Select the initial mode and apply a low-rank approximation to $\mathbf{T}_{\langle 1 \rangle}$  by \eqref{T1} 
            with the threshold $\delta_1$ \\
    	   Determine ranks $r_1$ and $r_2$: 
    		\[ \min_{r_1, r_2} || r_1 - r_2 || \quad \text{subject to} \quad r_1 r_2 = \text{rank}_{\delta_1}\left( \mathbf{T}_{\langle 1 \rangle} \right) \] \\
 	   Compute the first core tensor: 
    \[{ \mathcal{Z}}_1 = \text{permute}\left( \text{reshape}\left( \mathbf{U}_{1} , [I_1, r_1, r_2] \right), [2, 1, 3] \right) \] \\
           Derive another core tensor $\mathcal{Z}^{>1}$: 
    \[ \mathcal{Z}^{>1} = \text{permute}\left( \text{reshape}\left(  \mathbf{\Sigma}_{1} \mathbf{V}_{1}^H, [r_1, r_2, \prod_{j=2}^{N}I_j] \right), [2, 3, 1] \right)\]
                \For {$k=2$ to $N-1$}
                    \State  $\mathbf{Z}^{>k-1}= \text{reshape}\left(\mathcal{Z}^{>k-1}, [r_k I_k, \prod_{j=k+1}^{N}I_jr_1] \right)$
                    \State $\mathbf{Z}^{>k-1} = \mathbf{U}_{k} \mathbf{\Sigma}_{k} \mathbf{V}_{k}^H + \mathbf{E}_{k}$ 
		   \State $r_{k+1} = \text{rank}_{\delta_k}\left( \mathbf{Z}^{>k-1} \right)$ 
		   \State $\dot{ \mathcal{Z}}_k = \text{reshape}\left( \mathbf{U}_{k}, [r_k, I_k, r_{k+1}] \right)$ 
		   \State $ \mathcal{Z}^{>k} = \text{reshape}\left( \mathbf{\Sigma}_{k} \mathbf{V}_{k}^H , [r_{k+1}, \prod_{j=k+1}^{N}I_j ,r_1] \right)$ 
                \EndFor        
        \end{algorithmic}
    \end{algorithm}
\end{figure}

Finally, we validate the effectiveness of RBTR-SVD by comparing storage costs, compression ratios, PSNR, and RSE under various degrees of relative errors. For our analysis, we select ten color images\footnote{https://sipi.usc.edu/database/} with the size of $256\times256\times3$. Both of RBTR-SVD and TR-SVD are applied to these images and the comparative results are presented  in Table \ref{RBTR-SVD-result}. 

The storage cost is defined as \( \sum_{k=1}^{K} N_k\), where \(N_k\) is the number of elements in the \(k\)-th core tensor. The compression ratio is defined as $ \frac{N}{S} $, where $N$ is the total number of elements in the original tensor, and $S$ is the storage cost which equals the sum of the elements in the core tensors.

The relative error RSE measures the difference between the original tensor $\mathcal{X}$ and the recovered tensor $\hat{\mathcal{X}}$, which is 
\begin{flalign}
\mbox{RSE} = \frac{ \lVert \hat{\mathcal{X}}-\mathcal{X} \rVert_{F}  }{ \lVert \mathcal{X} \rVert_{F}}.
\label{RSE}
\end{flalign}

Define
\begin{flalign}
\mbox{PSNR} =10\log_{10}\left(\frac {\mbox{max}^2}{\lVert \hat{\mathcal{X}}-\mathcal{X} \rVert_{F} /N}\right),
\label{PSNR}
\end{flalign}
where $max$ equals  the original image data's maximum pixel value,  and $N$ represents the total number of elements in the tensor.

\begin{table}[!ht]
    \begin{center}
    \caption{RSE, PSNR, Storage Costs, and Compression Ratios (in parentheses) of the algorithms}
    \label{RBTR-SVD-result}
    \resizebox{\textwidth}{!}{%
    \begin{tabular}{cccccccc}
        \hline
        \multicolumn{1}{c}{\multirow{2}{*}{Relative Error}} & \multicolumn{1}{c}{\multirow{2}{*}{Image}} & \multicolumn{3}{c}{TR-SVD} & \multicolumn{3}{c}{RBTR-SVD} \\ 
        \cline{3-8} 
        & & RSE & PSNR & Storage Cost & RSE & PSNR & Storage Cost \\ \hline
        \multirow{10}{*}{0.005} 
        & Airplane & 2.91e-03 & 53.39 & 260784 (0.75) & \textbf{1.96e-03} & \textbf{56.82} & \textbf{96676 (2.03)} \\ 
        & Peppers & 2.37e-03 & 58.40 & 296748 (0.66) & \textbf{1.60e-03} & \textbf{61.82} & \textbf{98576 (1.99)} \\ 
        & Baboon & 2.13e-03 & 58.75 & 302385 (0.65) & \textbf{2.06e-14} & \textbf{279.07} & \textbf{100496 (1.96)} \\ 
        & Female & 1.69e-03 & 66.33 & 301737 (0.65) & \textbf{1.13e-14} & \textbf{289.85} & \textbf{100496 (1.96)} \\ 
        & House1 & 3.27e-03 & 54.32 & 266577 (0.74) & \textbf{2.21e-03} & \textbf{57.72} & \textbf{97320 (2.02)} \\ 
        & Tree & 2.45e-03 & 57.41 & 295544 (0.67) & \textbf{1.70e-03} & \textbf{60.58} & \textbf{99216 (1.98)} \\
        & Butterfly & 2.95e-03 & 56.31 & 263236 (0.75) & \textbf{1.91e-03} & \textbf{60.09} & \textbf{97936 (2.01)} \\ 
        & House2 & 3.08e-03 & 54.04 & 241873 (0.81) & \textbf{1.72e-03} & \textbf{59.09} & \textbf{98576 (1.99)} \\ 
        & Baby & 3.31e-03 & 53.63 & 205937 (0.95) & \textbf{3.23e-03} & \textbf{54.04} & \textbf{91812 (2.14)} \\ 
        & Bird & 3.12e-03 & 59.20 & 238109 (0.83) & \textbf{3.02e-03} & \textbf{59.49} & \textbf{89976 (2.19)} \\  \hline
        \multirow{10}{*}{0.015} 
        & Airplane & 1.05e-02 & 42.29 & 114105 (1.72) & \textbf{8.94e-03} & \textbf{43.64} & \textbf{81264 (2.42)} \\ 
        & Peppers & 8.61e-03 & 47.20 & 237677 (0.83) & \textbf{8.23e-03} & \textbf{47.59} & \textbf{88224 (2.23)} \\ 
        & Baboon & 8.52e-03 & 46.72 & 284832 (0.69) & \textbf{5.16e-03} & \textbf{51.08} & \textbf{98576 (1.99)} \\ 
        & Female & 7.80e-03 & 53.07 & 319600 (0.62) & \textbf{6.63e-03} & \textbf{54.48} & \textbf{97320 (2.02)} \\ 
        & House1 & 9.64e-03 & 44.94 & 209356 (0.94) & \textbf{9.05e-03} & \textbf{45.48} & \textbf{78544 (2.50)} \\ 
        & Tree & 1.03e-02 & 44.92 & 215125 (0.91) & \textbf{8.67e-03} & \textbf{46.42} & \textbf{90668 (2.17)} \\ 
        & Butterfly & 9.82e-03 & 45.87 & 167912 (1.17) & \textbf{9.13e-03} & \textbf{46.50} & \textbf{89392 (2.20)} \\ 
        & House2 & 8.75e-03 & 44.97 & 214032 (0.92) & \textbf{8.39e-03} & \textbf{45.35} & \textbf{87608 (2.24)} \\ 
        & Baby & 1.04e-02 & 43.87 & 117613 (1.67) & \textbf{1.01e-02} & \textbf{44.09} & \textbf{68304 (2.88)} \\ 
        & Bird & 9.98e-03 & 49.11 & 166137 (1.18) & \textbf{9.27e-03} & \textbf{49.75} & \textbf{70504 (2.79)} \\ \hline
        
        \multirow{10}{*}{0.025} 
        & Airplane & 1.80e-02 & 37.42 & 115384 (1.70) & \textbf{1.74e-02} & \textbf{37.84} & \textbf{65704 (2.99)} \\ 
        & Peppers & 1.66e-02 & 41.51 & 210600 (0.93) & \textbf{1.45e-02} & \textbf{42.65} & \textbf{77320 (2.54)} \\ 
        & Baboon & 1.62e-02 & 41.13 & 249149 (0.79) & \textbf{1.29e-02} & \textbf{43.11} & \textbf{94792 (2.07)} \\ 
        & Female & 1.43e-02 & 47.79 & 262504 (0.75) & \textbf{1.41e-02} & \textbf{47.95} & \textbf{91216 (2.17)} \\ 
        & House1 & 1.89e-02 & 39.11 & 141368 (1.41) & \textbf{1.74e-02} & \textbf{39.79} & \textbf{59176 (3.32)} \\ 
        &Tree & 1.50e-02 & 41.64 & 199016 (0.99) & \textbf{1.49e-02} & \textbf{41.74} & \textbf{81324 (2.42)} \\ 
        & Butterfly & 1.70e-02 & 41.08 & 125729 (1.56) & \textbf{1.61e-02} & \textbf{41.58} & \textbf{80528 (2.44)} \\ 
        & House2 & 1.82e-02 & 38.64 & 127381 (1.54) & \textbf{1.49e-02} & \textbf{40.34} & \textbf{75716 (2.60)} \\ 
        & Baby & 1.96e-02 & 38.42 & 81233 (2.42) & \textbf{1.94e-02} & \textbf{38.45} & \textbf{50376 (3.90)} \\ 
        & Bird & \textbf{1.68e-02} & \textbf{44.58} & 148500 (1.32) & 1.87e-02 & 43.65 & \textbf{55884 (3.52)} \\ \hline
    
        \multirow{10}{*}{0.05} 
        & Airplane & 3.95e-02 & 30.73 & 58146 (3.38) & \textbf{3.58e-02} & \textbf{31.60} & \textbf{41556 (4.73)} \\ 
        & Peppers & 3.45e-02 & 35.15 & 131364 (1.50) & \textbf{3.20e-02} & \textbf{35.81} & \textbf{53024 (3.71)} \\ 
        & Baboon & 3.43e-02 & 34.61 & 193380 (1.02) & \textbf{2.81e-02} & \textbf{36.37} & \textbf{84660 (2.32)} \\ 
        & Female & 3.95e-02 & 38.97 & 144033 (1.37) & \textbf{3.42e-02} & \textbf{40.23} & \textbf{71356 (2.76)} \\
        & House1 & 4.15e-02 & 32.26 & 50718 (3.88) & \textbf{3.74e-02} & \textbf{33.17} & \textbf{26532 (7.41)} \\ 
        & Tree & 3.40e-02 & 34.55 & 112576 (1.75) & \textbf{3.25e-02} & \textbf{34.94} & \textbf{58760 (3.35)} \\ 
        & Butterfly & 3.37e-02 & 35.17 & 96984 (2.03) & \textbf{3.29e-02} & \textbf{35.37} & \textbf{60880 (3.23)} \\ 
        & House2 & 3.91e-02 & 31.98 & 68053 (2.89) & \textbf{3.70e-02} & \textbf{32.46} & \textbf{48208 (4.08)} \\ 
        & Baby & \textbf{3.57e-02} & 31.98 & 52108 (3.77) & 4.09e-02 & 31.98 & \textbf{24072 (8.17)} \\ 
        & Bird & 3.34e-02 & 38.42 & 91864 (2.14) & \textbf{3.33e-02} & \textbf{38.64} & \textbf{35344 (5.56)} \\ 
        \hline
    \end{tabular}
    }
    \end{center}
\end{table}
Table \ref{RBTR-SVD-result} shows that, in most images, RBTR-SVD achieves higher PSNR and lower RSE compared to TR-SVD. Additionally, as the relative error level increases, the differences in PSNR and RSE between the two methods tend to decrease. However, RBTR-SVD offers a substantial reduction in storage cost at the same relative error level. Experimental results demonstrate that reduced biquaternion tensor ring decomposition is an effective representation method, reducing data size efficiently while retaining substantial information.

\section{Reduced biquaternion tensor completion}
As mentioned in Section 1, the low-rank reduced biquaternion tensor completion (LRRBTC) problem can be modeled as 
\begin{flalign}
 \min\limits_{\mathcal{X}} \ \text{rank}(\mathcal{X}), \  
 s.t. \quad P_{\Omega}(\mathcal{X})= P_{\Omega}(\mathcal{T}). 
\label{obj1}
\end{flalign}
We first define $\text{rank}(\mathcal{X})$ through  the RBTR decomposition. We will adopt the idea in 
 \cite{yu2019tensor} to  minimize the tensor circular unfolding rank as an alternative to the tensor rank.  

\begin{definition}(Reduced Biquaternion Tensor Circular Unfolding) \\
    Let ${\mathcal{X}}\in{\mathbb{H}_c}^{I_{1}\times I_{2}\times \dots \times I_{N}}$. Its circular unfolding $\mathbf{X}_{<k,d>}$ is defined as:
\begin{equation}
    \mathbf{X}_{<k,d>}(\overline{i_{m}i_{m+1}\dots i_{k}}, \overline{i_{k+1}\dots i_{m-1}}) = \mathcal{X}(i_{1}, i_{2}, \dots, i_{N}), \notag 
    \end{equation}
    where 
\begin{equation}
    m = \begin{cases}
    k-d+1, & \text{if } d \le k; \notag \\
    k-d+1+N, & \text{otherwise}.
    \end{cases}
\end{equation}
The rows of $\mathbf{X}_{<k,d>}$ are enumerated by the $d$ indices $\{i_{m}, i_{m+1}, \dots, i_{k}\}$, while its columns are enumerated by the remaining $N-d$ indices. In addition, $\mbox{fold}_{<k,d>}(\mathbf{X}_{<k,d>}) =  \mathcal{X}$, which represents the inverse process. 
\end{definition}
    
\begin{theorem}
Given an $N$th-order RB tensor  ${\mathcal{X}}\in{\mathbb{H}_c}^{I_{1}\times{I_{2}}\times{...}\times{I_{N}}}$ with a RBTR decomposition format $\mathfrak{TR}(\mathcal{Z}_1,\mathcal{Z}_2,\cdots, \mathcal{Z}_N)$ and its RBTR rank ${\mathbf{r}}=[r_{1},r_{2},\dots,r_{N}]$, we have
\begin{flalign}
\mbox{rank}(\mathbf{X}_{<k,d>}) \leq r_{k+1}r_{m}.\notag
\end{flalign}
\end{theorem}
\begin{proof}
\begin{flalign}
&\mathbf{X}_{<k,d>}(\overline{i_{m}i_{m+1}\cdots i_{k}}, \overline{i_{k+1}\cdots i_{m-1}}) \notag \\ 
&= \mathcal{X}(i_{1},i_{2},\cdots,i_{N}) \notag \\ 
&= \mathrm{Tr}\{\mathbf{Z}_{1}(i_{1})\mathbf{Z}_{2}(i_{2})\cdots \mathbf{Z}_{N}(i_{N})\}    \notag \\
&= \mathrm{Tr}\{\mathbf{Z}_{m}(i_{m}) \cdots \mathbf{Z}_{k}(i_{k})\mathbf{Z}_{k+1}(i_{k+1}) \cdots \mathbf{Z}_{m-1}(i_{m-1})\}  \notag \\ 
&=\mathrm{Tr}\{\mathbf{A}(\overline{i_{m} i_{m+1} \cdots i_{k}}) \mathbf{B}(\overline{i_{k+1} \cdots i_{m-1}})\}, \notag
\end{flalign}
where \( \mathcal{A} \in \mathbb{H}_c^{r_{m} \times I_m I_{m+1} \cdots I_k \times r_{k+1}}\) with 
\[ \mathbf{A}(\overline{i_m i_{m+1} \cdots i_k}) = \mathbf{Z}_m(i_m) \mathbf{Z}_{m+1}(i_{m+1}) \cdots \mathbf{Z}_k(i_k)\]
and \( \mathcal{B}  \in \mathbb{R}^{r_{k+1} \times I_{k+1}I_{k+2} \ldots I_{m-1} \times r_{m}}\) with 
\[ \mathbf{B}(\overline{i_{k+1} i_{k+2} \cdots i_{m-1}}) = \mathbf{Z}_{k+1}(i_{k+1}) \mathbf{Z}_{k+2}(i_{k+2}) \cdots \mathbf{Z}_{m-1}(i_{m-1}).\]
Then
\begin{flalign}
&\mathbf{X}_{<k,d>}(\overline{i_{m}i_{m+1}\cdots i_{k}}, \overline{i_{k+1}\cdots i_{m-1}}) \notag \\ 
&= \sum_{\gamma_{k+1}=1}^{r_{k+1}} \sum_{\gamma_{m}=1}^{r_{m}} \mathcal{A} (\gamma_{m}, \overline{i_m i_{m+1} \cdots i_k}, \gamma_{k+1})\mathcal{B} (\gamma_{k+1}, \overline{i_{k+1} \cdots i_{m-1}}, \gamma_{m}) \notag \\
&= \sum_{\gamma_{k+1}\gamma_{m}=1}^{r_{k+1} r_{m}} \mathbf{A}_{(2)}(\overline{i_m i_{m+1} \cdots i_k}, \gamma_{k+1} \gamma_{m})  \mathbf{B}_{[2]}^T(\gamma_{k+1} \gamma_{m}, \overline{i_{k+1} \cdots i_{m-1}}). \notag
\end{flalign}
Based on the above  processes, we have 
$
\mathbf{X}_{<k,d>}=\mathbf{A}_{(2)}\mathbf{B}_{[2]}^T$.
Therefore, by using (f) of Theorem \ref{pro22}, 
\begin{align*}
&\text{rank}(\mathbf{X}_{<k,d>}) \leq\min(\text{rank}(\mathbf{A}_{(2)}),\text{rank}(\mathbf{B}_{[2]}^T))\\&=\frac{1}{4}\min(\text{rank}((\mathbf{A}_{(2)})^R),\text{rank}((\mathbf{B}_{[2]}^T)^R), 
\end{align*}
that is, $\text{rank}(\mathbf{X}_{<k,d>}) \leq r_{k+1}r_{m}$ as desired.
\end{proof}

Now, we can transform the LRRBTC problem \eqref{obj1} into the following problem
\begin{equation}
\label{np}
\min\limits_{\mathcal{X}} \quad \sum_{k=1}^{N} \alpha_k \text{rank}(\mathbf{X}_{<k,d>}), \ \
\mbox{s.t.} \ \ \quad P_{\Omega}(\mathcal{X})=P_{\Omega}(\mathcal{T}), 
\end{equation}
where $\alpha_k>0$ is a scalar weight and  $\sum_{k=1}^{N} \alpha_k = 1$. 

%\subsection{Reduced Biquaternion Tensor Ring Nuclear Norm}
However, the optimization problem as formulated in \eqref{np} is NP-hard  \cite{hillar2013most}. To address this challenge, the researchers recently have turned to alternative strategies. Among these, minimizing the nuclear norm has emerged as a popular choice. The nuclear norm offers a convex relaxation for the non-convex rank minimization problem, providing an effective approximation \cite{jia2019robust}. Therefore we propose the reduced biquaternion tensor ring nuclear norm.

\begin{definition}(RBTR Nuclear Norm)\\
Let ${\mathcal{X}}\in{\mathbb{H}_c}^{I_{1}\times{I_{2}}\times{...}\times{I_{N}}}$ with a RBTR format. We define its RBTR nuclear norm as:
\begin{flalign}
&&
\mbox{NN}_{{RBTR}} (\mathcal{X}) =\sum_{k=1}^{N} \alpha_{k} \lVert \mathbf{X}_{<k,d>} \rVert_{*},
&&
\end{flalign}
where $\alpha_k$'s  are the same as ones in \eqref{np}.
\end{definition}

%\subsection{TV Regularization}
In many representative tensor completion methods currently in use, a technique known as key augmentation has been widely adopted. This approach primarily aims to better mine the data's low-rank structures, as it enhances the representational capability and flexibility towards the original tensor \cite{bengua2017efficient}. However, a notable downside is that employing tensor augmentation can introduce block artifacts, manifested as visible block-like distortions in the data.

To overcome this challenge and improve the visual quality of the data, the researchers have proposed the incorporation of total variation (TV) regularization \cite{ding2019low}. TV regularization is a mathematical technique designed to smooth the values within an image or tensor, while preserving its primary structures and features. The underlying idea is to penalize large local variations within the tensor, promoting data continuity and smoothness. As a result, TV regularization can effectively reduce block artifacts caused by tensor augmentation, thus enhancing the performance and quality of the completion task. The isotropic TV is defined by
\begin{flalign}
    \mbox{TV}(\mathcal{X})=\sum_{m=1}^{I_{1}} \sum_{n=1}^{t} \sqrt {|D^{1}_{m,n} \mathbf{X}_{(1)}|^2 + |D^{2}_{m,n} \mathbf{X}_{(1)}|^2},  
\end{flalign}
where $\mathbf{X}_{(1)} \in  {\mathbb{H}_c}^{I_1 \times t}, \ t={\prod_{i=2}^{N}I_{i}}$. At the $(m, n)$-th pixel within the matrix $\mathbf{X}_{(1)}$, $D^{1}_{m,n} \mathbf{X}_{(1)}$ represents the gradient in the horizontal axis and  $D^{2}_{m,n} \mathbf{X}_{(1)}$ represents the gradient in the vertical axis. Correspondingly, the operators $D^{1}_{m,n}$ and $D^{2}_{m,n}$ denote the discrete gradients in the horizontal and vertical directions.

%\subsection{The Proposed Algorithm}
By utilizing RBTR nuclear norm and TV regularization, our RBTR-TV method is proposed as:
\begin{flalign}
\min\limits_{\mathcal{X}}  \quad &\sum_{k=1}^{N} \alpha_{k} \lVert \mathbf{X}_{<k,d>} \rVert_{*}+\lambda \mbox{TV}(\mathcal{X}) \notag\\[3mm]
s.t. \quad &P_{\Omega}(\mathcal{X})=P_{\Omega}(\mathcal{T}), 
\label{obj2}
\end{flalign}
where $\lambda$ is the regularization parameter. 

To  solve the \eqref{obj2}, we first introduce auxiliary reduced biquaternion variables $\mathcal{A}_k (k = 1,2, \dots, N)$ and $\mathcal{Z}$  by using the variable-splitting technique as follows:
\begin{flalign}
\min\limits_{\mathcal{X},{\mathcal{A}_k},\mathcal{Z}}  \quad &\sum_{k=1}^{N} \alpha_{k} \lVert \mathbf{A}_{k_{<k,d>}} \rVert_{*}+\lambda \sum_{m=1}^{I_{1}} \sum_{n=1}^{t}  \lVert \mathbf{E}_{m,n} \rVert_{2} \notag\\[3mm]
s.t. \quad &P_{\Omega}(\mathcal{X})=P_{\Omega}(\mathcal{T}), \notag \\
& \mathcal{X}=\mathcal{A}_k,\mathcal{X}=\mathcal{Z},  \notag \\
& \mathbf{D}_1\mathbf{Z}_{(1)}=\mathbf{E}_{1},\mathbf{D}_2\mathbf{Z}_{(1)}=\mathbf{E}_{2}, \notag
\end{flalign}
where $\mathbf{Z}_{(1)}$ is the classical mode-1 unfolding of $\mathcal{Z}$,  $\mathbf{E}_{m,n}=[(\mathbf{E}_{1})_{m,n},(\mathbf{E}_{2})_{m,n}]$ and $\mathbf{E}_{1},\mathbf{E}_{2}$ denote the outcomes of applying horizontal and vertical discrete gradient matrices $\mathbf{D}_1$ and $\mathbf{D}_2$ to $\textbf{Z}_{(1)}$, respectively.

The augmented Lagrangian function is defined using the ADMM framework as follows: 
\begin{flalign}
    &\mathcal{L}(\mathcal{X},\{\mathcal{A}_k \}_{k=1}^{N},\mathcal{Z}, \{ \mathbf{E}_{i} \}_{i=1}^{2},\{\mathcal{B}_k \}_{k=1}^{N},{\mathcal{Q}},\{ \mathbf{F}_{i} \}_{i=1}^{2})  \notag \\
    &=  \sum_{k=1}^{N} \alpha_{k} \lVert \mathbf{A}_{k_{<k,d>}} \rVert_{*}+\lambda \sum_{m=1}^{I_{1}} \sum_{n=1}^{t}  \lVert \mathbf{E}_{m,n} \rVert_{2}  \notag \\ 
    & +\sum_{k=1}^{N}(\frac{\beta_{1}}{2}\lVert  {\mathcal{X}} - {\mathcal{A}_k} \rVert^{2}_{F} + 
    \Re (<{\mathcal{X}} - {\mathcal{A}_k},{\mathcal{B}_k}>))  \notag \\
    & +\frac{\beta_{2}}{2}\lVert  {\mathcal{X}} - {\mathcal{Z}} \rVert^{2}_{F} +
    \Re ( <{\mathcal{X}} - {\mathcal{Z}},{\mathcal{Q}}>)   \notag \\
    & +\sum_{i=1}^{2}(\frac{\beta_{3}}{2}\lVert \mathbf{D}_i\mathbf{Z}_{(1)} -  \mathbf{E}_{i} \rVert^{2}_{F} + \Re (<\mathbf{D}_i\mathbf{Z}_{(1)} -  \mathbf{E}_{i}, \mathbf{F}_{i}>)), \notag
    \end{flalign}
where $\mathcal{B}_{k} (k = 1,2, \dots, N)$, ${\mathcal{Q}}$ and $\textbf{F}_{i}(i=1,2)$ are Lagrange multipliers, $\beta_{i}> 0(i=1,2,3)$ is the penalty parameter, and $p$ represents the number of iterations. Next, we solve the problem using the iterative approach described below:
\begin{align}
    \left \{
    \begin{aligned}
    \mathcal{X}^{p+1} &=\mathop{\mathrm{argmin}}\limits_{P_{\Omega}(\mathcal{X}) = P_{\Omega}(\mathcal{T})} \mathcal{L}(\mathcal{X}, \{\mathcal{A}_k^{p} \}_{k=1}^{N}, \mathcal{Z}^{p}, \{\mathcal{B}_{k}^{p} \}_{k=1}^{N}, \mathcal{Q}^{p}) \\
   \mathcal{A}_k^{p+1} &=\mathop{\mathrm{argmin}}\limits_{\mathcal{A}_k} \mathcal{L} (\mathcal{X}^{p+1},\mathcal{A}_k,\mathcal{B}_{k}^{p}) \\
    \mathcal{Z}^{p+1}  &=\mathop{\mathrm{argmin}}\limits_{\mathcal{Z}} \mathcal{L}(\mathcal{X}^{p+1} ,\mathcal{Z},\{ \textbf{E}_{i}^p \}_{i=1}^{2},\mathcal{Q}^{p},\{ \textbf{F}_{i}^p \}_{i=1}^{2}) \\
    \mathbf{E}_{i}^{p+1} &=\mathop{\mathrm{argmin}}\limits_{\mathbf{E}_{i}} \mathcal{L}(\mathcal{Z}^{p+1}, \mathbf{E}_{i}, \mathbf{F}_{i}^{p}) \\
   \mathcal{B}_{k}^{p+1} &=\mathcal{B}_{k}^{p} + \beta_{1}(\mathcal{X}^{p+1} -\mathcal{A}_{k}^{p+1})\\
    \mathcal{Q}^{p+1} &= \mathcal{Q}^{p} + \beta_{2}(\mathcal{X}^{p+1} - \mathcal{Z}^{p+1}) \\
    \mathbf{F}_{i}^{p+1} &= \mathbf{F}_{i}^{p} + \beta_{3}(\mathbf{D}_i\mathbf{Z}_{(1)}^{p+1} - \mathbf{E}_{i}^{p+1}).
    \end{aligned}
    \right. \notag
\end{align}

\noindent $\textbf{Updating } \mathcal{X}: $
\begin{flalign*}
\mathcal{X}^{p+1} &= \mathop{\mathrm{argmin}}\limits_{P_{\Omega}(\mathcal{X})=P_{\Omega}(\mathcal{T})} \mathcal{L}(\mathcal{X},\{\mathcal{A}_k^{p} \}_{k=1}^{N},\mathcal{Z}^{p},\{\mathcal{B}_{k}^{p} \}_{k=1}^{N},\mathcal{Q}^{p})   \\
    \quad 
    &= \mathop{\mathrm{argmin}}\limits_{P_{\Omega}(\mathcal{X})=P_{\Omega}(\mathcal{T})} \sum_{k=1}^{N}\left(\frac{\beta_{1}}{2}\lVert  {\mathcal{X}} - {{\mathcal{A}_k^{p}}} \rVert^{2}_{F} + \Re \left(<{\mathcal{X}} - {{\mathcal{A}_k^{p}}},{\mathcal{B}_k^{p}}>\right)\right)  \\
    &\quad + \frac{\beta_{2}}{2}\lVert  {\mathcal{X}} - {\mathcal{Z}}^{p} \rVert^{2}_{F} + \Re \left(<{\mathcal{X}} - {\mathcal{Z}}^{p},{\mathcal{Q}}^{p}>\right).
\end{flalign*} 
For this optimization problem, adding a constant unrelated to ${\mathcal{X}}$ does not affect the optimization outcome. Therefore, we transform the problem into: 
\begin{flalign}
\mathcal{X}^{p+1} 
    &= \mathop{\mathrm{argmin}}\limits_{P_{\Omega}(\mathcal{X})=P_{\Omega}(\mathcal{T})} \sum_{k=1}^{N}\left(\frac{\beta_{1}}{2}\left(\lVert  {\mathcal{X}} - {{\mathcal{A}_k^{p}}} \rVert^{2}_{F} + 2\Re \left(<{\mathcal{X}} - {{\mathcal{A}_k^{p}}},\frac{\mathcal{B}_k^{p}}{\beta_1}>\right) + \lVert \frac{\mathcal{B}_k^{p}}{\beta_1} \rVert_F^2\right) \right)  \notag  \\[3mm]
    &\quad + \frac{\beta_{2}}{2} \left( \lVert  {\mathcal{X}} - {\mathcal{Z}}^{p} \rVert^{2}_{F} + 2 \Re \left(<{\mathcal{X}} - {\mathcal{Z}}^{p},\frac{{\mathcal{Q}}^{p}}{\beta_2}>\right) + \frac{{\mathcal{Q}}^{p}}{\beta_2}\right)    \notag \\[3mm]
    &= \mathop{\mathrm{argmin}}\limits_{P_{\Omega}(\mathcal{X})=P_{\Omega}(\mathcal{T})} \sum_{k=1}^{N}\left(\frac{\beta_{1}}{2}\lVert  {\mathcal{X}} - {{\mathcal{A}_k^{p}}} + \frac{ {\mathcal{B}_k^{p}}}{\beta_{1}} \rVert^{2}_{F}\right) + \frac{\beta_{2}}{2}\lVert  {\mathcal{X}} - {\mathcal{Z}^{p}} + \frac{{\mathcal{Q}}^{p}}{\beta_{2}}\rVert^{2}_{F}   \notag \\[3mm]
    \quad &= P_{{\Omega}^{c}}\left(\frac{\sum_{k=1}^{N}( \beta_{1} {{\mathcal{A}_k^{p}}} - {\mathcal{B}_k^{p}} ) + \beta_{2} {\mathcal{Z}^{p}} -{\mathcal{Q}^{p}} }{N \beta_{1}+\beta_{2}}\right) + P_{\Omega} (\mathcal{T}),\label{X}
\end{flalign}
where  $\Omega^c$ represents the complement of $\Omega$, which contains all the elements that are not in $ \Omega$. 

\noindent $\textbf{Updating } \mathcal{A}_k: $
\begin{flalign}
\mathcal{A}_k^{p+1} 
&= \mathop{\mathrm{argmin}}\limits_{\mathcal{A}_k} \mathcal{L} (\mathcal{X}^{p+1},\mathcal{A}_k,\mathcal{B}_{k}^{p}) \notag \\
&=\mathop{\mathrm{argmin}}\alpha_{k}   \lVert \mathbf{A}_{k_{<k,d>}} \rVert_{*} +\frac{\beta_{1}}{2}\lVert  {\mathcal{X}}^{p+1} - {\mathcal{A}_k} \rVert^{2}_{F}  + \Re ( <{\mathcal{X}}^{p+1} - {\mathcal{A}_k},\mathcal{B}_{k}^{p}>) \notag \\
&= \mathop{\mathrm{argmin}} \alpha_{k}  \lVert \mathbf{A}_{k_{<k,d>}} \rVert_{*} +\frac{\beta_{1}}{2}\lVert  {\mathcal{X}}^{p+1} - {\mathcal{A}_k} + \frac{\mathcal{B}_k^{p}}{\beta_{1}}  \rVert^{2}_{F} \notag \\
&= \mathop{\mathrm{argmin}} \alpha_{k}  \lVert \mathbf{A}_{k_{<k,d>}} \rVert_{*} +\frac{\beta_{1}}{2}\lVert  {\mathbf{X}}^{p+1}_{<k,d>} - \mathbf{A}_{k_{<k,d>}} + \frac{{\mathbf{B}_k^p}_{<k,d>}}{\beta_{1}}  \rVert^{2}_{F}. \notag
\end{flalign}
 
Denote $\bm{\Gamma}= {\mathbf{X}}^{p+1}_{<k,d>} + \frac{{\mathbf{B}_k^p}_{<k,d>}}{\beta_{1}}$  and let $\bm{\Gamma} = \textbf{U} \boldsymbol{\Sigma} \textbf{V}^H ,\tau=\frac{\alpha_{k}}{\beta_{1}}$. Thus  $\mathcal{A}_k^{p+1}$ has the closed-form solution,  which follows a proof process analogous to that in the quaternion domain \cite{chen2019low}, and it can be presented as: 
\begin{flalign}
\mathcal{A}_k^{p+1} = fold_{<k,d>}(\textbf{U} S_{\tau}(\boldsymbol{\Sigma})\textbf{V}^H ),
\label{A}
\end{flalign}
where $S_{\tau}(\boldsymbol{\Sigma})= \mbox{diag}(\max(0,{\sigma_i(\bm{\Gamma})-\tau}))$.

%更新Z
\noindent $\textbf{Updating $\mathcal{Z}$}: $
\begin{flalign}
\mathcal{Z}^{p+1} 
&= \mathop{\mathrm{argmin}} \limits_{\mathcal{Z}} \mathcal{L}(\mathcal{X}^{p+1} ,\mathcal{Z},\{ \textbf{E}_{i}^p \}_{i=1}^{2},\mathcal{Q}^{p},\{ \textbf{F}_{i}^p \}_{i=1}^{2}) \notag \\
&= \mathop{\mathrm{argmin}}  \frac{\beta_{2}}{2}\lVert  {\mathcal{X}}^{p+1} - {\mathcal{Z}} \rVert^{2}_{F} +  \Re (<{\mathcal{X}}^{p+1} - {\mathcal{Z}},{\mathcal{Q}}^{p} >)  \notag \\
& +\sum_{i=1}^{2}(\frac{\beta_{3}}{2}\lVert \mathbf{D}_i\textbf{Z}_{(1)} -  \textbf{E}_{i}^p \rVert^{2}_{F} + \Re (<\mathbf{D}_i\textbf{Z}_{(1)} -  \textbf{E}_{i}^p, \textbf{F}_{i}^p>)) \notag \\
&= \mathop{\mathrm{argmin}}  \frac{\beta_{2}}{2}\lVert  {\mathcal{X}}^{p+1} - {\mathcal{Z}} + \frac{{\mathcal{Q}}^{p}}{\beta_{2}} \rVert^{2}_{F} + \sum_{i=1}^{2}\left(\frac{\beta_{3}}{2}\lVert \mathbf{D}_i\textbf{Z}_{(1)} -  \textbf{E}_{i}^p + \frac{\textbf{F}_{i}^p}{\beta_{3}} \rVert^{2}_{F}\right).\notag 
\end{flalign}

Since the tensor's Frobenius norm and the tensor's unfolded matrix's Frobenius norm are the same, the aforementioned problem is transformed into the following formula:
\begin{flalign}
\mathbf{Z}_{(1)}^{p+1}  = \mathop{\mathrm{argmin}}  \frac{\beta_{2}}{2}\lVert  {\mathbf{X}}_{(1)}^{p+1} - {\mathbf{Z}}_{(1)} + \frac{{\mathbf{Q}}_{(1)}^{p}}{\beta_{2}} \rVert^{2}_{F} +\sum_{i=1}^{2}\left(\frac{\beta_{3}}{2}\lVert \mathbf{D}_i\mathbf{Z}_{(1)} - \mathbf{E}_{i}^{p} + \frac{{\mathbf{F}_i}^{p}}{\beta_{3}} \rVert^{2}_{F}\right). \notag
\end{flalign}
The problem can be equivalent to the following formula:
\begin{flalign}
\mathbf{A} \mathbf{Z}_{(1)}^{p+1} = \mathbf{B}, \notag
\end{flalign}
where 
\[
\mathbf{A}=\beta_{2}\mathbf{I} +\beta_{3}(\mathbf{D}_1^T\mathbf{D}_1+\mathbf{D}_2^T\mathbf{D}_2), \ \mathbf{B}= \beta_2 {\mathbf{X}}_{(1)}^{p+1} + {\mathbf{Q}}_{(1)}^{p} + \beta_3 \mathbf{D}_1^T \mathbf{E}_{1}^{p}  - \mathbf{D}_1^T \mathbf{F}_{1}^{p} + \beta_3 \mathbf{D}_2^T \mathbf{E}_{2}^{p}  - \mathbf{D}_2^T \mathbf{F}_{2}^{p}.
\]
As demonstrated in  \cite{pei2004commutative}, while the spatial domain convolution of two quaternion matrices cannot be computed through the multiplication of their Fourier transforms in the frequency domain, the convolution operation for RB matrices in the spatial domain does equate to their multiplication in the frequency domain, which is similar to the properties of convolution for real-valued matrices. Therefore the problem can be solved quickly by using the Fourier transform \cite{ding2019low}:
\begin{flalign}
\mathbf{Z}_{(1)}^{p+1} = {\mathcal{F}}^{-1} \left(\frac{{\mathcal{F}}(\mathbf{\textbf{B}})}{{\mathcal{F}}(\textbf{A})}\right).
\label{Z}
\end{flalign}

\noindent $\textbf{Updating $\mathbf{E}_{i}$}: $
\begin{flalign}
\mathbf{E}_{i}^{p+1}
&= \mathop{\mathrm{argmin}}\limits_{\mathbf{E}_{i}} \mathcal{L}(\mathcal{Z}^{p+1},\mathbf{E}_{i},\mathbf{F}_{i}^{p}) \notag \\
&= \mathop{\mathrm{argmin}}\limits_{\mathbf{E}_{i}}  \lambda \sum_{m=1}^{I_{1}} \sum_{n=1}^{t}  \lVert \mathbf{E}_{m,n} \rVert_{2} +\sum_{i=1}^{2}(\frac{\beta_{3}}{2}\lVert \mathbf{D}_i\mathbf{Z}_{(1)}^{p+1} -  \mathbf{E}_{i} \rVert^{2}_{F} + \Re (<\mathbf{D}_i\mathbf{Z}_{(1)}^{p+1} -  \mathbf{E}_{i}, \mathbf{F}_{i}^p>)) \notag \\
&=  \mathop{\mathrm{argmin}} \limits_{\mathbf{E}_{i}}  \lambda \sum_{m=1}^{I_{1}} \sum_{n=1}^{t}  \lVert \mathbf{E}_{m,n} \rVert_{2} + \sum_{i=1}^{2}\left(\frac{\beta_{3}}{2}\lVert \mathbf{D}_i\mathbf{Z}_{(1)}^{p+1} - \mathbf{E}_{i} + \frac{{\mathbf{F}_i^{p}}}{\beta_{3}} \rVert^{2}_{F}\right). \notag
\end{flalign}
In order to find the optimal $\mathbf{E}_{i}$, one must solve $I_t$ independent minimization problems as follows due to the objective function being the sum of squared terms of each element in $\mathbf{E}_{i}$, which are independent of one another:
\begin{flalign*}
\mathop{\mathrm{argmin}} \limits_{({\mathbf{E}_{1}},{\mathbf{E}_{2}})_{m,n}}  \lambda \sqrt{|(\mathbf{E}_{1})_{m,n}|^2+|(\mathbf{E}_{2})_{m,n}|^2} + \frac{\beta_{3}}{2}\sum_{i=1}^{2} \left| (\mathbf{D}_i\mathbf{Z}_{(1)}^{p+1})_{m,n} - (\mathbf{E}_{i})_{m,n} + \frac{({\mathbf{F}_i^{p}})_{m,n}}{\beta_{3}} \right|^{2}. 
\end{flalign*}
As proven in Theorem \ref{min2norm}, we have:
\begin{flalign}
[(\mathbf{E}_{1})_{m,n},(\mathbf{E}_{2})_{m,n}] = \max \left\{ \lVert \mathbf{w}_{m,n} \rVert_{2} - \frac{\lambda}{\beta_{3}} ,\ 0\right\} \frac{\mathbf{w}_{m,n}}{\lVert \mathbf{w}_{m,n} \rVert_{2}}, \label{E}
\end{flalign}
where 
$$\mathbf{w}_{m,n} = \left[(\mathbf{D}_1\mathbf{Z}_{(1)}^{p+1})_{m,n} + \frac{({\mathbf{F}_1^{p}})_{m,n}}{\beta_{3}}, \ (\mathbf{D}_2\mathbf{Z}_{(1)}^{p+1})_{m,n} + \frac{({\mathbf{F}_2^{p}})_{m,n}}{\beta_{3}}\right] \ ({1 \le m } \le I_{1} ,{1 \le n } \le t). $$

%更新BQF
\noindent $\textbf{Updating $\mathcal{B}_{k},\mathcal{Q},\textbf{F}_{i}$}: $
\begin{flalign}
\mathcal{B}_{k}^{p+1} &=\mathcal{B}_{k}^{p} + \beta_{1}(\mathcal{X}^{p+1}-\mathcal{A}_{k}^{p+1}), \label{B} \\
\mathcal{Q}^{p+1} &= \mathcal{Q}^{p} + \beta_{2}(\mathcal{X}^{p+1}-\mathcal{Z}^{p+1}), \label{Q} \\
\mathbf{F}_{i}^{p+1} &= \mathbf{F}_{i}^{p} + \beta_{3}(\mathbf{D}_i\mathbf{Z}_{(1)}^{p+1} - \mathbf{E}_{i}^{p+1}). \label{F}
\end{flalign}

Ultimately, we design the following $\textbf{Algorithm 2}$ which is a summary of the low-rank RB tensor completion approach, leveraging the RB tensor ring decomposition and the total variation regularization.
\begin{figure}[ht]
    \centering  % Centers the algorithm in the figure
    \begin{algorithm}[H]
        \caption{RBTR-TV}
        \begin{algorithmic}[1]
            \Require 
                The observed RB tensor $\mathcal{T} \in {\mathbb{H}_c}^{I_{1}\times I_{2}\times \dots \times I_{N}}$, 
                the index set ${\Omega}$,
                the parameters: $\alpha_k$ for $k=1,2,...,N$, the regularization parameters: $\lambda, \beta_1, \beta_2, \beta_3$, and the maximum number of iterations: $maxIter$.
        \Ensure The recovered RB tensor $\mathcal{X}^{p}$.
        \vspace{3mm}

            \State Initialize: $\{\mathcal{A}_k^{0}  \}_{k=1}^{N} , \{\mathcal{B}_k^{0}  \}_{k=1}^{N} , \mathcal{Z}^{0} , \{\mathbf{E}_{i}^{0} \}_{i=1}^{2}, \mathcal{Q}^0, \{ \mathbf{F}_{i}^0 \}_{i=1}^{2}, p=0$\vspace{.2cm}
            % Main Loop
            \While { $\left( \frac{ \lVert \mathcal{X}^{p+1}-\mathcal{X}^{p} \rVert_{F}  }{ \lVert \mathcal{X}^{p} \rVert_{F}} > 10^{-5} \right)$ or $(p < maxIter)$ } \vspace{.2cm}
                
                \State Update $\mathcal{X}^{p+1}$ by using \eqref{X}
                
                \For{$k=1$ to $N$}
                    \State Update $\mathcal{A}_k^{p+1}$ by using \eqref{A}
                \EndFor
                
                \State Update $\mathcal{Z}^{p+1}$ by using \eqref{Z}
                \State Update $\textbf{E}_i^{p+1}$ by using \eqref{E}
                 
                \For{$k=1$ to $N$}
                    \State Update $\mathcal{B}_k^{p+1}$ by using \eqref{B}
                \EndFor
                
                \State Update $\mathcal{Q}^{p+1}$ by using \eqref{Q}
                \State Update $\mathbf{F}_i^{p+1}$ by using \eqref{F}
               
                \State Increment $p$: $p \leftarrow p+1$
            \EndWhile

        \end{algorithmic}
    \end{algorithm}
    \label{alg:RBTR-TV}
\end{figure}

\section{Experiment}
\subsection{Color image completion}
In this section, we compare our proposed RBTR-TV with six baselines in color image completion experiments, including SiLRTC-TT\cite{bengua2017efficient}, LRQC\cite{chen2023quaternion},  MF-TTTV\cite{ding2019low}, TVTRC\cite{he2019total}, RTRC\cite{huang2020robust},   and QTT-SRTD\cite{miao2024quaternion}. All methods tune parameters according to literature references.

In order to quantitatively compare the completion performance, we use PSNR (see \eqref{PSNR}), RSE (see \eqref{RSE}) and Time (in seconds) as metrics. Ten color images, each represented as a $256 \times 256 \times 3$ real tensor, are used to compare RBTR-TV with other methods. In addition, we use ket augmentation technology and change each image into a $4\times 4 \times 4 \times 4 \times 4 \times 4 \times 4 \times 4 \times 3$ high-order real tensor. After converting to RBs representation, the dimension is $4\times 4 \times 4 \times 4 \times 4 \times 4 \times 4 \times 4$. We set $\alpha_k = 0.125, \lambda = 0.3, \beta_1 = 5 \times 10^{-3}, \beta_2 = 0.1, \text{ and } \beta_3 = 5 \times 10^{-3}$.

Table \ref{image_result} shows the performance comparison results at sampling rates SR={5\%, 10\%, 15\%, 20\%}. The method RBTR-TV consistently outperforms other comparative methods both in PSNR and RSE across all sampling rates. These results indicate the robustness and effectiveness of RBTR-TV in handling various levels of data sparsity. Figure \ref{image_result_table} is the visual content based on $SR=20\%$.

% 开始表格
\begin{table}
\centering
\scalebox{0.95}{
\caption{PSNR, RSE and Time of various methods with four sampling rates (bold indicates best performance)}\label{image_result}
\renewcommand{\arraystretch}{1.1}
\resizebox{\textwidth}{!}{%
\begin{tabular}{cccccccccccccc}
\hline
Image & SR & \multicolumn{3}{c}{5\%} & \multicolumn{3}{c}{10\%} & \multicolumn{3}{c}{15\%} & \multicolumn{3}{c}{20\%} \\ \cline{3-5} \cline{6-8} \cline{9-11} \cline{12-14}
& Method & PSNR & RSE & Time & PSNR & RSE & Time & PSNR & RSE & Time & PSNR & RSE & Time \\ \hline
& SiLRTC-TT & 18.28 & 0.1697 & 10.63 & 19.79 & 0.1437 & 15.29 & 20.97 & 0.1260 & 13.71 & 22.03 & 0.1120 & 11.31 \\
& RTRC & 16.33 & 0.2145 & 29.61 & 20.35 & 0.1369 & 30.57 & 21.92 & 0.1139 & 24.58 & 23.73 & 0.0933 & 28.48 \\
& MF-TTTV & 20.06 & 0.1350 & 362.91 & 22.67 & 0.1000 & 345.75 & 24.25 & 0.0834 & 253.74 & 25.44 & 0.0727 & 189.10 \\
& TVTRC & 10.98 & 0.3841 & 36.71 & 12.73 & 0.3139 & 37.06 & 16.65 & 0.2001 & 38.29 & 20.01 & 0.1358 & 27.20 \\
& LRQC & 16.84 & 0.1957 & \textbf{7.55} & 21.05 & 0.1206 & \textbf{7.53} & 22.70 & 0.0996 & \textbf{5.54} & 23.94 & 0.0864 & \textbf{4.61} \\
& QTT-SRTD & 19.82 & 0.1355 & 46.06 & 21.50 & 0.1106 & 44.19 & 23.07 & 0.0948 & 45.14 & 24.26 & 0.0791 & 41.29 \\
\multirow{-7}{*}{Airplane} & RBTR-TV & \textbf{22.05} & \textbf{0.1113} & 12.91 & \textbf{24.04} & \textbf{0.0900} & 11.43 & \textbf{25.40} & \textbf{0.0756} & 11.80 & \textbf{26.78} & \textbf{0.0645} & 11.17 \\ \hline
& SiLRTC-TT & 15.44 & 0.3408 & 15.21 & 18.97 & 0.2249 & 15.39 & 20.57 & 0.1868 & 14.06 & 21.81 & 0.1623 & 11.39 \\
& RTRC & 12.88 & 0.4677 & 29.28 & 18.39 & 0.2478 & 28.19 & 21.42 & 0.1716 & 29.33 & 22.43 & 0.1523 & 29.89 \\
& MF-TTTV & 18.63 & 0.2309 & 347.56 & 22.86 & 0.1419 & 368.07 & 24.77 & 0.1139 & 347.05 & 26.17 & 0.0969 & 286.79 \\
& TVTRC & 9.34 & 0.6733 & 36.34 & 14.30 & 0.3804 & 37.00 & 17.66 & 0.2582 & 34.80 & 18.77 & 0.2273 & 24.75 \\
& LRQC & 15.96 & 0.3140 & \textbf{6.35} & 19.51 & 0.2087 & \textbf{6.68} & 22.60 & 0.1687 & \textbf{5.23} & 24.04 & 0.1422 & \textbf{5.21} \\
& QTT-SRTD & 20.79 & 0.1748 & 54.38 & 22.97 & 0.1341 & 53.10 & 24.74 & 0.1148 & 57.66 & 25.83 & 0.0991 & 50.92 \\
\multirow{-7}{*}{Peppers} & RBTR-TV & \textbf{22.91} & \textbf{0.1440} & 11.50 & \textbf{25.16} & \textbf{0.1089} & 11.40 & \textbf{26.88} & \textbf{0.0905} & 11.45 & \textbf{28.30} & \textbf{0.0774} & 11.58 \\ \hline
& SiLRTC-TT & 15.71 & 0.3045 & 16.49 & 17.58 & 0.2456 & 16.04 & 18.66 & 0.2163 & 15.10 & 19.39 & 0.1989 & 12.50 \\
& RTRC & 13.02 & 0.4153 & 27.62 & 17.80 & 0.2589 & 28.21 & 19.60 & 0.1980 & 25.75 & 20.43 & 0.1852 & 33.09 \\
& MF-TTTV & 16.36 & 0.2809 & 408.78 & 18.71 & 0.2141 & 349.43 & 19.86 & 0.1876 & 361.14 & 20.72 & 0.1700 & 259.97 \\
& TVTRC & 10.84 & 0.5298 & 36.44 & 13.40 & 0.3946 & 36.37 & 16.24 & 0.2846 & 38.19 & 18.25 & 0.2258 & 41.30 \\
& LRQC & 16.21 & 0.2857 & \textbf{6.77} & 17.97 & 0.2332 & \textbf{5.88} & 19.10 & 0.2048 & \textbf{5.01} & 19.99 & 0.1849 & \textbf{4.35} \\
& QTT-SRTD & 18.42 & 0.2178 & 57.46 & 19.65 & 0.1878 & 56.35 & 20.90 & 0.1647 & 61.12 & 21.73 & 0.1260 & 54.67 \\
\multirow{-7}{*}{Baboon} & RBTR-TV & \textbf{19.97} & \textbf{0.1868} & 12.42 & \textbf{21.33} & \textbf{0.1596} & 13.49 & \textbf{22.39} & \textbf{0.1415} & 12.03 & \textbf{23.44} & \textbf{0.1253} & 13.13 \\ \hline
& SiLRTC-TT & 19.43 & 0.3574 & 15.78 & 22.05 & 0.2617 & 14.69 & 23.63 & 0.2172 & 11.29 & 24.96 & 0.1855 & 9.86 \\
& RTRC & 16.79 & 0.4842 & 19.13 & 23.25 & 0.2284 & 19.09 & 25.33 & 0.1789 & 19.74 & 26.35 & 0.1586 & 19.21 \\
& MF-TTTV & 21.80 & 0.2853 & 350.61 & 26.50 & 0.1661 & 345.74 & 28.27 & 0.1355 & 266.60 & 29.86 & 0.1129 & 244.84 \\
& TVTRC & 13.02 & 0.7846 & 12.38 & 18.24 & 0.4301 & 40.62 & 20.76 & 0.3215 & 37.36 & 22.95 & 0.2502 & 22.32 \\
& LRQC & 15.90 & 0.5627 & \textbf{6.38} & 22.54 & 0.2620 & \textbf{7.06} & 25.66 & 0.1830 & \textbf{6.09} & 27.48 & 0.1485 & \textbf{5.50} \\
& QTT-SRTD & 24.22 & 0.1917 & 71.04 & 26.17 & 0.1509 & 70.10 & 27.61 & 0.1267 & 75.67 & 28.75 & 0.1104 & 75.32 \\
\multirow{-7}{*}{Female} & RBTR-TV & \textbf{26.35} & \textbf{0.1691} & 12.27 & \textbf{28.67} & \textbf{0.1294} & 11.43 & \textbf{30.31} & \textbf{0.1071} & 11.51 & \textbf{31.69} & \textbf{0.0915} & 11.57 \\ \hline

& SiLRTC-TT & 16.70 & 0.2120 & 14.26 & 18.83 & 0.1656 & \textbf{15.35} & 20.33 & 0.1393 & 14.25 & 21.60 & 0.1203 & 11.32 \\
& RTRC & 14.36 & 0.2785 & 20.06 & 18.04 & 0.1921 & 29.13 & 20.46 & 0.1396 & 26.10 & 22.48 & 0.1118 & 27.28 \\
& MF-TTTV & 19.92 & 0.1566 & 344.68 & 23.07 & 0.1090 & 344.32 & 24.65 & 0.0909 & 291.64 & 26.03 & 0.0775 & 208.95 \\
& TVTRC & 11.13 & 0.4311 & 39.24 & 16.78 & 0.2249 & 36.47 & 18.12 & 0.1927 & 40.02 & 20.55 & 0.1457 & 23.73 \\
& LRQC & 17.65 & 0.2034 & \textbf{6.59} & 21.46 & 0.1311 & \textbf{7.57} & 23.01 & 0.1097 & \textbf{5.65} & 24.32 & 0.0943 & \textbf{4.45} \\
& QTT-SRTD & 20.53 & 0.1418 & 47.91 & 22.52 & 0.1114 & 47.89 & 23.49 & 0.0932 & 47.89 & 25.31 & 0.0795 & 46.63 \\
\multirow{-7}{*}{House1} & RBTR-TV & \textbf{21.95} & \textbf{0.1240} & 11.42 & \textbf{23.98} & \textbf{0.0981} & 11.44 & \textbf{25.54} & \textbf{0.0820} & 11.36 & \textbf{26.86} & \textbf{0.0705} & 11.22 \\ \hline

& SiLRTC-TT & 16.34 & 0.2671 & 16.51 & 18.69 & \textbf{0.2053} & \textbf{15.41} & 20.14 & 0.1750 & 14.24 & 21.33 & 0.1528 & 11.52 \\
& RTRC & 13.44 & 0.3790 & 29.80 & 19.59 & 0.1892 & 21.83 & 21.14 & 0.1562 & 29.22 & 22.65 & 0.1322 & 27.37 \\
& MF-TTTV & 18.16 & 0.2243 & 344.49 & 22.61 & 0.1344 & 383.65 & 24.24 & 0.1114 & 344.23 & 25.85 & 0.0925 & 296.94 \\
& TVTRC & 10.80 & 0.5234 & 36.21 & 13.94 & 0.3649 & 37.16 & 18.63 & 0.2127 & 40.96 & 19.80 & 0.1857 & 16.23 \\
& LRQC & 16.59 & 0.2687 & \textbf{6.55} & 20.19 & 0.1776 & \textbf{5.76} & 22.27 & 0.1397 & \textbf{4.97} & 23.46 & 0.1219 & \textbf{4.22} \\
& QTT-SRTD & 19.84 & 0.1805 & 53.72 & 21.74 & 0.1417 & 50.23 & 23.40 & 0.1152 & 50.23 & 24.61 & 0.1000 & 54.27 \\
\multirow{-7}{*}{Tree} & RBTR-TV & \textbf{21.42} & \textbf{0.1541} & 11.49 & \textbf{23.67} & \textbf{0.1189} & 13.70 & \textbf{25.50} & \textbf{0.0964} & 11.52 & \textbf{26.75} & \textbf{0.0834} & 11.58 \\ \hline

& SiLRTC-TT & 13.72 & 0.3984 & 15.07 & 15.83 & 0.3103 & 15.14 & 17.57 & 0.2534 & 15.93 & 18.99 & 0.2158 & 13.97 \\
& RTRC & 12.41 & 0.4771 & 22.32 & 16.07 & 0.3158 & 38.72 & 18.83 & 0.2227 & 27.50 & 20.49 & 0.1885 & 20.09 \\
& MF-TTTV & 16.98 & 0.2735 & 362.65 & 20.42 & 0.1841 & 362.66 & 21.93 & 0.1491 & 369.05 & 24.15 & 0.1223 & 249.03 \\
& TVTRC & 7.83 & 0.7840 & 38.53 & 9.23 & 0.6672 & 38.11 & 12.82 & 0.4413 & 38.29 & 14.59 & 0.3602 & 16.72 \\
& LRQC & 16.55 & 0.3392 & \textbf{7.29} & 18.98 & 0.2505 & \textbf{5.22} & 20.85 & 0.2046 & \textbf{4.33} & 21.99 & 0.1736 & \textbf{3.86} \\
& QTT-SRTD & 17.01 & 0.2702 & 59.37 & 19.27 & 0.2048 & 59.41 & 21.06 & 0.1568 & 57.78 & 22.57 & 0.1370 & 57.70 \\
\multirow{-7}{*}{Butterfly} & RBTR-TV & \textbf{18.50} & \textbf{0.2288} & 12.35 & \textbf{21.21} & \textbf{0.1669} & 13.12 & \textbf{22.75} & \textbf{0.1397} & 12.91 & \textbf{24.47} & \textbf{0.1145} & 13.16 \\ \hline

& SiLRTC-TT & 19.11 & 0.1766 & 16.51 & 21.52 & 0.1347 & 15.48 & 23.19 & 0.1109 & 12.90 & 24.42 & 0.0962 & 10.48 \\
& RTRC & 15.63 & 0.2703 & 22.28 & 22.20 & 0.1260 & 22.60 & 24.11 & 0.0999 & 22.58 & 26.30 & 0.0773 & 27.24 \\
& MF-TTTV & 24.38 & 0.1021 & 342.00 & 26.32 & 0.0772 & 307.20 & 27.36 & 0.0657 & 171.98 & 29.56 & 0.0551 & 130.78 \\
& TVTRC & 12.87 & 0.3865 & 40.20 & 16.50 & 0.2544 & 43.21 & 19.62 & 0.1778 & 44.91 & 20.08 & 0.1686 & 42.19 \\
& LRQC & 19.62 & 0.1896 & \textbf{5.73} & 24.40 & 0.1104 & \textbf{6.07} & 26.57 & 0.0869 & \textbf{5.54} & 27.55 & 0.0764 & \textbf{4.56} \\
& QTT-SRTD & 22.90 & 0.1120 & 50.45 & 25.24 & 0.0835 & 53.58 & 26.90 & 0.0695 & 49.67 & 28.23 & 0.0593 & 46.26 \\
\multirow{-7}{*}{House2} & RBTR-TV & \textbf{24.64} & \textbf{0.0954} & 12.63 & \textbf{26.84} & \textbf{0.0742} & 12.92 & \textbf{28.54} & \textbf{0.0613} & 12.89 & \textbf{29.90} & \textbf{0.0525} & 12.76 \\ \hline

& SiLRTC-TT & 19.02 & 0.1877 & 14.76 & 22.34 & 0.1279 & 14.91 & 24.06 & 0.1041 & 13.73 & 25.42 & 0.0890 & 10.50 \\
& RTRC & 13.94 & 0.3310 & 24.80 & 23.82 & 0.1075 & 19.58 & 25.40 & 0.0888 & 19.03 & 27.31 & 0.0712 & 21.43 \\
& MF-TTTV & 23.64 & 0.1169 & 344.60 & 26.12 & 0.0788 & 248.00 & 28.94 & 0.0608 & 174.40 & 29.86 & 0.0546 & 135.17 \\
& TVTRC & 12.15 & 0.4011 & 43.30 & 18.17 & 0.2007 & 46.31 & 21.35 & 0.1391 & 21.78 & 22.70 & 0.1191 & 15.36 \\
& LRQC & 18.21 & 0.2445 & \textbf{6.33} & 24.00 & 0.1144 & \textbf{6.82} & 26.84 & 0.0802 & \textbf{5.83} & 28.90 & 0.0634 & \textbf{4.69} \\
& QTT-SRTD & 23.47 & 0.1095 & 48.10 & 25.68 & 0.0828 & 47.04 & 27.04 & 0.0713 & 44.08 & 28.10 & 0.0623 & 45.23 \\
\multirow{-7}{*}{Baby} & RBTR-TV & \textbf{25.01} & \textbf{0.0928} & 13.71 & \textbf{27.35} & \textbf{0.0701} & 12.60 & \textbf{29.29} & \textbf{0.0561} & 12.35 & \textbf{30.75} & \textbf{0.0474} & 13.60 \\ \hline
% Image: Bird
& SiLRTC-TT & 17.34 & 0.3849 & 14.82 & 19.80 & 0.2901 & 15.19 & 21.91 & 0.2271 & 12.58 & 23.34 & 0.1929 & 10.19 \\
& RTRC & 14.48 & 0.5543 & 23.29 & 19.21 & 0.3204 & 20.28 & 22.96 & 0.2111 & 29.95 & 24.02 & 0.1825 & 27.89 \\
& MF-TTTV & 20.35 & 0.2795 & 339.42 & 25.19 & 0.1582 & 260.60 & 28.76 & 0.1029 & 244.13 & 28.81 & 0.0965 & 126.61 \\
& TVTRC & 12.13 & 0.7051 & 39.97 & 15.30 & 0.4896 & 40.87 & 18.49 & 0.3389 & 49.88 & 20.21 & 0.2781 & 23.45 \\
& LRQC & 15.85 & 0.4842 & \textbf{6.95} & 21.45 & 0.2674 & \textbf{7.26} & 24.85 & 0.1976 & \textbf{5.77} & 25.91 & 0.1599 & \textbf{4.72} \\
& QTT-SRTD & 23.34 & 0.1836 & 64.26 & 26.23 & 0.1294 & 62.94 & 28.36 & 0.1044 & 61.37 & 29.85 & 0.0835 & 63.41 \\
\multirow{-7}{*}{Bird} & RBTR-TV & \textbf{24.98} & \textbf{0.1591} & 12.81 & \textbf{27.95} & \textbf{0.1129} & 13.07 & \textbf{30.15} & \textbf{0.0876} & 12.67 & \textbf{31.83} & \textbf{0.0722} & 12.64 \\ \hline
\end{tabular}%
}}
\end{table}
%\FloatBarrier
%\begin{figure}
%    \centering
%    \includegraphics{image.PNG}
%\end{figure}

%\begin{figure*}[!ht]
\begin{figure*}
        \centering
        \includegraphics[width=\textwidth]{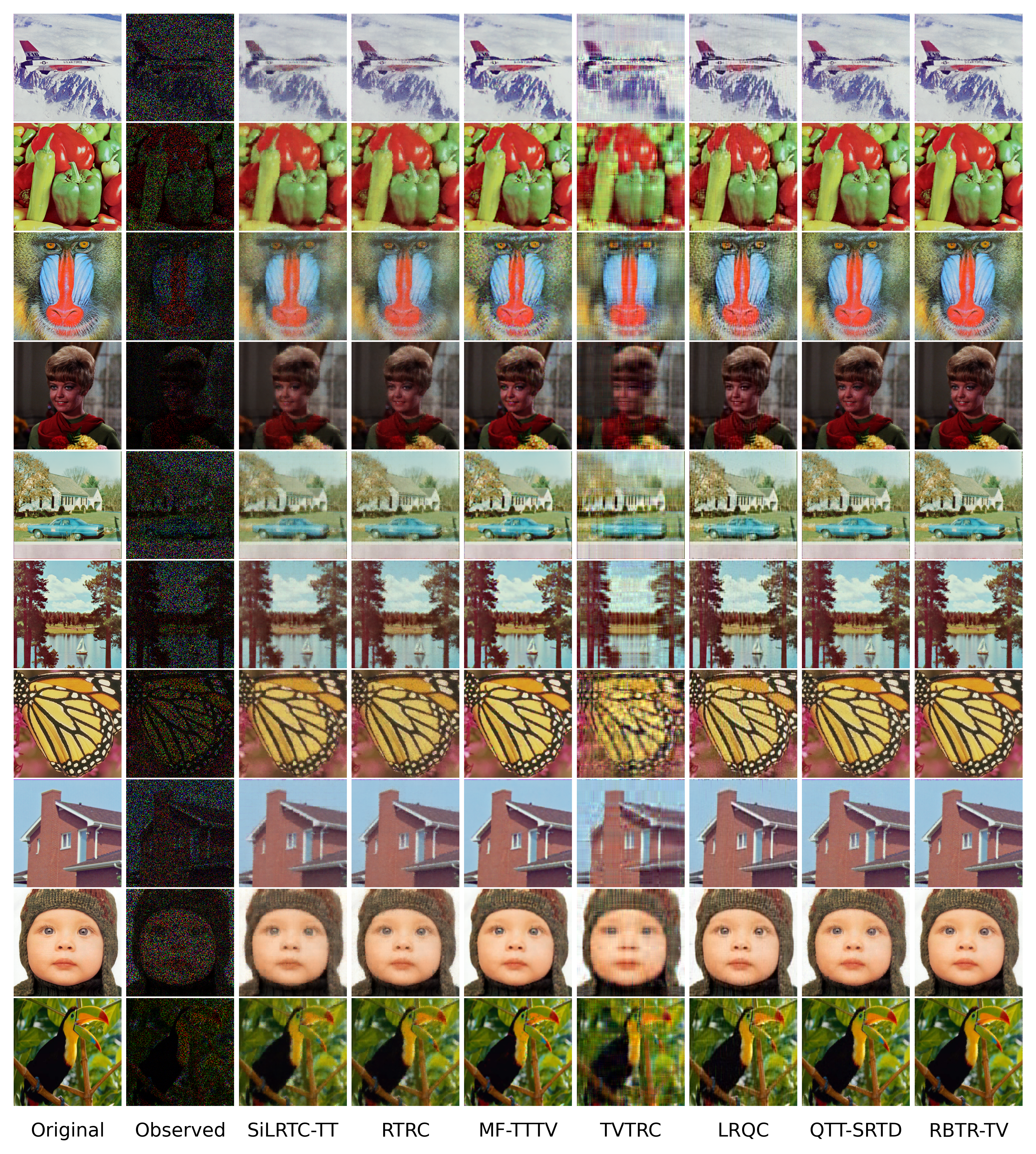}
        \caption{Recovery performance of different methods at SR=20\%}
        \label{image_result_table}
\end{figure*}

\subsection{Color video completion}
In this experiment, we select four videos\footnote{http://trace.eas.asu.edu/yuv/} for testing: \textit{bus}, \textit{foreman}, \textit{mother}, and \textit{hall} with SR=10\% and SR=20\%. Each video was sampled for 15 frames with a size of $256\times 256$. The methods we compared include SiLRTC-TT\cite{bengua2017efficient},  MF-TTTV\cite{ding2019low}, TVTRC\cite{he2019total}, RTRC\cite{huang2020robust},   QTT-SRTD\cite{miao2024quaternion}, and our proposed RBTR-TV.

Table \ref{video_result} presents the average PSNR, RSE, and Time (in seconds) over all frames for each method. Our proposed RBTR-TV method achieves the best PSNR and RSE values among all the compared methods.
Although it is not the fastest in terms of running time, it significantly improves the recovered quality compared to the fastest method. Figure \ref{video_plot} shows the recovery results of frame 1 and frame 15 for the \textit{foreman} and \textit{bus} videos at SR=10\%. It can be seen that our proposed method performs better in recovering missing data, especially when a large amount of information is missing.

\begin{table}[!ht]
\centering
\caption{PSNR, RSE and Time of various methods with two sampling rates (bold indicates best performance)}\label{video_result}
\resizebox{0.8\textwidth}{!}{%
\begin{tabular}{cccccccc}
\hline
Video &
SR &
  \multicolumn{3}{c}{10\%} &
  \multicolumn{3}{c}{20\%} \\ \cline{3-8} 
 &
 Method &
  PSNR &
  RSE &
  Time &
  PSNR &
  RSE &
  Time \\ \hline
 &
  SiLRTC-TT &
  18.95 &
  0.3383 &
  \textbf{4.28} &
  21.62 &
  0.2515 &
  \textbf{4.57} \\
 &
  RTRC &
  20.91 &
  0.26083 &
  69.13 &
  23.62 &
  0.2038 &
  62.56 \\
 &
  MF-TTTV &
  21.33 &
  0.2683 &
  81.08 &
  23.55 &
  0.2049 &
  56.07 \\
 &
  TVTRC &
  23.03 &
  0.2440 &
  97.44 &
  23.74 &
  0.2185 &
  57.54 \\
 &
  QTT-SRTD &
  23.76 &
  0.1876 &
  23.73 &
  27.52 &
  0.1349 &
  23.72 \\
\multirow{-6}{*}{Bus} &
  RBTR-TV &
  \textbf{25.17} &
  \textbf{0.1692} &
  18.17 &
  \textbf{28.31} &
  \textbf{0.1156} &
  18.27 \\ \hline
 &
  SiLRTC-TT &
  20.37 &
  0.14407 &
  \textbf{4.19} &
  25.14 &
  0.0840 &
  \textbf{4.20} \\
 &
  RTRC &
  25.29 &
  0.0848 &
  64.41 &
  29.75 &
  0.0512 &
  70.57 \\
 &
  MF-TTTV &
  25.45 &
  0.0803 &
  67.54 &
  27.82 &
  0.0612 &
  39.53 \\
 &
  TVTRC &
  26.51 &
  0.08151 &
  35.87 &
  28.67 &
  0.0594 &
  26.32 \\
 &
  QTT-SRTD &
  27.85 &
  0.0506 &
  23.95 &
  31.13 &
  0.0343 &
  23.68 \\
\multirow{-6}{*}{Foreman} &
  RBTR-TV &
  \textbf{33.53} &
  \textbf{0.0318} &
  19.86 &
  \textbf{36.70} &
  \textbf{0.0220} &
  19.38 \\ \hline
 &
  SiLRTC-TT &
  25.84 &
  0.0895 &
  \textbf{4.06} &
  29.58 &
  0.0579 &
  \textbf{4.93} \\
 &
  RTRC &
  27.61 &
  0.0738 &
  61.25 &
  32.51 &
  0.0415 &
  60.44 \\
 &
  MF-TTTV &
  28.18 &
  0.0682 &
  66.91 &
  30.68 &
  0.0511 &
  39.59 \\
 &
  TVTRC &
  33.32 &
  0.0384 &
  61.74 &
  32.49 &
  0.0417 &
  13.00 \\
 &
  QTT-SRTD &
  32.03 &
  0.0344 &
  25.32 &
  35.28 &
  0.0230 &
  25.01 \\
\multirow{-6}{*}{Mother} &
  \textbf{RBTR-TV} &
  \textbf{38.38} &
  \textbf{0.0211} &
  18.66 &
  \textbf{42.11} &
  \textbf{0.0136} &
  18.50 \\ \hline
 &
  SiLRTC-TT &
  21.94 &
  0.1438 &
  \textbf{4.06} &
  26.38 &
  0.0873 &
  \textbf{4.27} \\
 &
  RTRC &
  26.86 &
  0.0874 &
  61.93 &
  30.59 &
  0.0539 &
  71.81 \\
 &
  MF-TTTV &
  26.06 &
  0.0883 &
  71.19 &
  29.27 &
  0.0616 &
  41.45 \\
 &
  TVTRC &
  31.26 &
  0.0528 &
  27.01 &
  32.22 &
  0.0474 &
  11.01 \\
 &
  QTT-SRTD &
  32.68 &
  0.0336 &
  23.57 &
  36.19 &
  0.0219 &
  23.46 \\
\multirow{-6}{*}{Hall} &
  RBTR-TV &
  \textbf{37.79} &
  \textbf{0.0238} &
  19.34 &
  \textbf{40.30} &
  \textbf{0.0171} &
  19.02 \\ \hline
\end{tabular}%
}
\end{table}

\begin{figure*}[!ht]
        \centering
        \includegraphics[width=\textwidth]{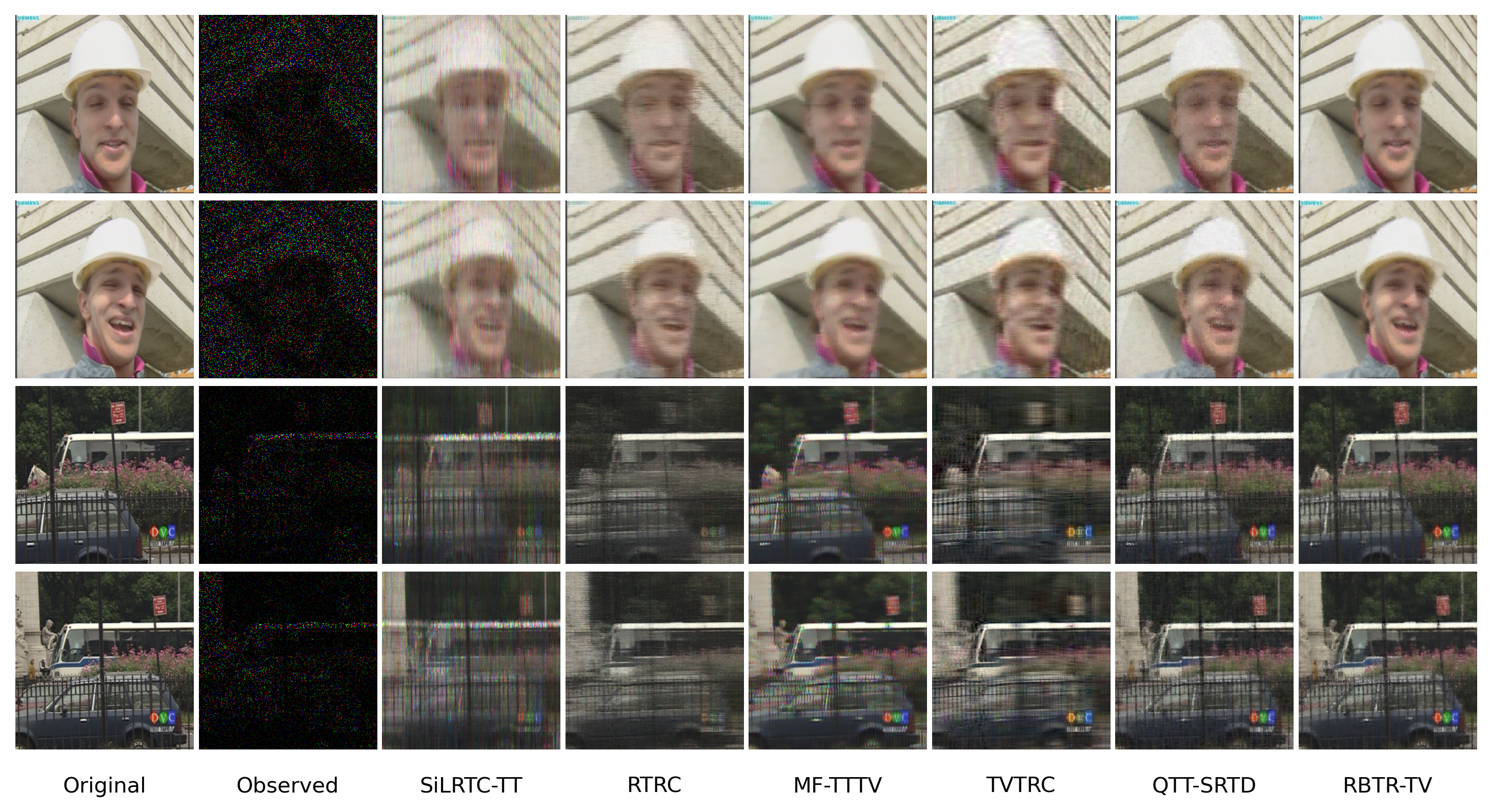}
        \caption{Recovery performance of different methods at SR=10\%}
        \label{video_plot}
\end{figure*}

\section{Conclusion and future works}

In this work, we propose the reduced biquaternion tensor ring (RBTR) decomposition and present its corresponding algorithm, which reduces storage costs compared to TR-SVD while preserving reconstruction quality. Building on the RBTR decomposition, we further introduce RBTR-TV, a novel low-rank tensor completion method that integrates RBTR rank with total variation (TV) regularization. Experimental results demonstrate its effectiveness in completing color images and videos.

Although RBTR-TV has shown strong performance in color image and video completion, there still remain some challenges. For example, the convergence speed of our current algorithm requires improvement.  Our future work will focus on strengthening the theoretical foundation and enhancing the algorithm's efficiency and speed.
\\
{\bf Acknowledgement:} This research is partially supported by the National Natural Science Foundation of China (12371023, 12271338), and Canada NSERC, The joint research and Development fund of Wuyi University, Hong Kong and Macao (2019WGALH20).
\iffalse
{\bf Declaration of Competing Interest:} The authors declare no conflict of interest.
\\

 \fi
 \section{Declarations}
{\bf Conflict of Interest:} The authors have not disclosed any competing interests.
\\
\\
{\bf Data Availability Statement:}
The data supporting the findings of this study are not publicly available due to privacy or ethical restrictions. However, interested researchers may request access to the data by contacting the corresponding author and completing any necessary data sharing agreements.

\end{document}